\documentclass[11pt]{amsart}
\usepackage{amsthm,amsmath,amssymb}
\usepackage{stmaryrd,mathrsfs}
\usepackage{enumerate}
\usepackage[T1]{fontenc}
\usepackage{esint}
\usepackage{tikz}

\allowdisplaybreaks

\newcommand{\ud}[0]{\,\mathrm{d}}

\newcommand{\abs}[1]{|#1|}

\newcommand{\Norm}[2]{\|#1\|_{#2}}

\newcommand{\ave}[1]{\langle #1\rangle}

\newcommand{\BMO}[0]{\operatorname{BMO}}
\newcommand{\CMO}[0]{\operatorname{CMO}}
\newcommand{\VMO}[0]{\operatorname{VMO}}
\newcommand{\supp}[0]{\operatorname{supp}}
\newcommand{\loc}[0]{\operatorname{loc}}

\newcommand{\R}{\mathbb{R}}
\newcommand{\C}{\mathbb{C}}
\newcommand{\N}{\mathbb{N}}

\newcommand{\eps}[0]{\varepsilon}

\swapnumbers \numberwithin{equation}{section}

\theoremstyle{plain}
\newtheorem{theorem}[equation]{Theorem}
\newtheorem{proposition}[equation]{Proposition}
\newtheorem{corollary}[equation]{Corollary}
\newtheorem{lemma}[equation]{Lemma}
\newtheorem{conjecture}[equation]{Conjecture}

\theoremstyle{definition}
\newtheorem{definition}[equation]{Definition}

\newtheorem{remark}[equation]{Remark}

\makeatletter
\@namedef{subjclassname@2020}{%
  \textup{2020} Mathematics Subject Classification}
 \def\@textbottom{\vskip \z@ \@plus 1pt}
 \let\@texttop\relax
\makeatother

\begin{document}

\title[Extrapolation of compactness]{Extrapolation of compactness on\\ weighted spaces}

\author[T. \ Hyt\"onen and S. Lappas]{Tuomas Hyt\"onen and Stefanos Lappas}
\address{Department of Mathematics and Statistics, P.O.Box~68 (Pietari Kalmin katu~5), FI-00014 University of Helsinki, Finland}
\email{tuomas.hytonen@helsinki.fi}
\email{stefanos.lappas@helsinki.fi}





\maketitle


\begin{abstract}
Let $T$ be a linear operator that, for some $p_1\in(1,\infty)$, is bounded on $L^{p_1}(\tilde w)$ for all $\tilde w\in A_{p_1}(\R^d)$ and in addition compact on $L^{p_1}(w_1)$ for some $w_1\in A_{p_1}(\R^d)$. Then $T$ is bounded and compact on $L^p(w)$ for all $p\in(1,\infty)$ and all $w\in A_p(\R^d)$. This ``compact version'' of Rubio de Francia's celebrated weighted extrapolation theorem follows from a combination of results in the interpolation and extrapolation theory of weighted spaces on the one hand, and of compact operators on abstract spaces on the other hand. Moreover, generalizations of this extrapolation of compactness are obtained for operators that are bounded from one space to a different one (``off-diagonal estimates'') or only in a limited range of the $L^p$ scale. As applications, we easily recover several recent results on the weighted compactness of commutators of singular integral operators, fractional integrals and pseudo-differential operators, and obtain new results about the weighted compactness of commutators of Bochner--Riesz multipliers.
\end{abstract}

\section{Introduction}

By a {\em weight} we understand a locally integrable function $w\in L^1_{\loc}(\R^d)$ that is positive almost everywhere. As we will work in the weighted setting, we define the weighted Lebesgue spaces
\begin{equation*}
  L^p(w):=\Big\{f:\R^d\to\C\text{ measurable }\Big|\ \Norm{f}{L^p(w)}:=\Big(\int_{\R^d}\abs{f}^p w\Big)^{1/p}<\infty\Big\}
\end{equation*}
and we recall the definitions of $A_p(\R^d)$,
$A_{p,q}(\R^d)$, and $RH_r(\R^d)$ classes of weights first introduced by Muckenhoupt \cite{Muckenhoupt:Ap}, Muckenho\-upt--Wheeden \cite{MW:74}, and Gehring \cite{Gehring}, respectively:

\begin{definition}\label{def:Muchenhoupt weights}
A weight $w\in L_{\loc}^1(\R^d)$ is called a Muckenhoupt $A_p(\R^d)$ weight (or $w\in A_p(\R^d)$) if 
\begin{equation*}
\begin{split}
  &[w]_{A_p}:=\sup_Q\ave{w}_Q\ave{w^{-\frac{1}{p-1}}}_Q^{p-1}<\infty,\qquad 1<p<\infty, \\
  &[w]_{A_1}:=\sup_Q\ave{w}_Q\Norm{w^{-1}}{L^\infty(Q)}<\infty,\qquad p=1,
\end{split}
\end{equation*}
where the supremum is taken over all cubes $Q\subset\R^d$, and $\ave{w}_Q:=\abs{Q}^{-1}\int_Q w$. A weight $w$ is called an $A_{p,q}(\R^d)$ weight (or $w\in A_{p,q}(\R^d)$) if
\begin{equation*}
  [w]_{A_{p,q}}:=\sup_Q\ave{w^q}_Q^{1/q}\ave{w^{-p'}}_Q^{1/p'}<\infty,\qquad 1<p\leq q<\infty,
\end{equation*}
where $p':=p/(p-1)$ denotes the conjugate exponent.

We say that $w$ belongs to the reverse H\"older class $RH_r(\R^d)$ (or $w\in RH_r(\R^d)$) if
\begin{equation*}
\begin{split}
  &[w]_{RH_r}:=\sup_Q\ave{w^r}_Q^{1/r}\ave{w}_Q^{-1}<\infty,\qquad 1<r<\infty,  \\
  &[w]_{RH_{\infty}}:=\sup_{Q}\Norm{w}{L^\infty(Q)}\ave{w}_Q^{-1}<\infty,\qquad r=\infty.
\end{split}
\end{equation*}
\end{definition}
The role of these weights in Analysis is well recognised since the pioneering works \cite{CF, Gehring, HMW, Muckenhoupt:Ap}. In particular, 
the classes $A_p(\R^d)$ and $A_{p,q}(\R^d)$ were introduced to study the weighted norm inequalities for the {\em Hardy--Littlewood maximal function} and for {\em fractional integral operators}, respectively; see \cite{Muckenhoupt:Ap,MW:74}. On the other hand, the reverse H\"older classes $RH_r(\R^d)$ were introduced to study the {\em $L^p$-integrability} of the partial derivatives of a quasiconformal mapping; see \cite{Gehring}.

The following theorem of Rubio de Francia \cite{Rubio:factorAp} on the extrapolation of {\em boundedness} on weighted spaces is one of the most useful and powerful tools in the theory of weighted norm inequalities: (See also \cite{Garcia:extrapolation} for a constructive proof, \cite{DGPP} for a quantitative formulation, \cite{CUMP:book} for an extensive treatment of related matters, and \cite{Krylov:survey} for a recent survey on applications of this result to elliptic and parabolic equations.)

\begin{theorem}[\cite{Rubio:factorAp}]\label{thm:RdF}
Let $1\leq \lambda<p_1<\infty$, and let $T$ be a linear operator simultaneously defined and bounded on $L^{p_1}(\tilde w)$ for {\bf all} $\tilde w\in A_{p_1/\lambda}(\R^d)$, with the operator norm dominated by some increasing function of $[\tilde w]_{A_{p_1/\lambda}}$.
Then $T$ is also defined and bounded on $L^p(w)$ for all $p\in(\lambda,\infty)$ and all $w\in A_{p/\lambda}(\R^d)$.
\end{theorem}

In this paper, we provide a variant for extrapolation of {\em compactness}: (See Theorems \ref{thm:Off-dig.extrp.compact} and \ref{thm:limited range extrp.compact} in Section \ref{sec:extension} for extension to off-diagonal and limited range extrapolation of compactness.)

\begin{theorem}\label{thm:RdFcompact}
In the setting of Theorem \ref{thm:RdF}, suppose in addition that $T$ is compact on $L^{p_1}(w_1)$ for {\bf some} $w_1\in A_{p_1/\lambda}(\R^d)$. Then $T$ is compact on $L^p(w)$ for all $p\in(\lambda,\infty)$ and all $w\in A_{p/\lambda}(\R^d)$.
\end{theorem}

When $w_1\equiv 1$, Theorem \ref{thm:RdFcompact} says, roughly, that unweighted compactness bootstraps to weighted compactness, if weighted boundedness is already known. This is probably the most relevant case for most applications, like Corollary \ref{cor:ClopCruz} below (see also Sections \ref{sec:RSIO}--\ref{sec:Ap(varphi) weights}).

The paper is organized as follows: in Section \ref{sec:extension} we state Theorems \ref{thm:Off-dig.extrp.compact} and \ref{thm:limited range extrp.compact} which are extensions of Theorem \ref{thm:RdFcompact}.
The proofs of these results are presented in Section \ref{sec:main results} by collecting some previously known results and taking some auxiliary results for granted. Section \ref{sec:prop} is dedicated to the proofs of these auxiliary results (see Propositions \ref{prop:LpvInterm}, \ref{prop:offdiagonal} and \ref{prop:limitedrange}). In Sections \ref{sec:CZO}--\ref{comm. br. mult.} we provide several applications of our main results. In particular, we obtain previously known results for the commutators of singular integral operators, fractional integrals and a new result for the commutators of {\em Bochner--Riesz multipliers}. In Section \ref{sec:Ap(varphi) weights} we develop and apply yet another variant for extrapolation of compactness for a special class of weights related to the commutators of {\em pseudo-differential operators} with smooth symbols.

\subsection*{Notation} Throughout the paper, we denote by $C$ a positive constant which is independent of the main parameters but it may change at each occurrence, and we write $f\lesssim g$ if $f\leq Cg$. The term cube will always refer to a cube $Q\subset\R^d$ and $|Q|$ will denote its Lebesgue measure. Recall from Definition \ref{def:Muchenhoupt weights} that $\ave{w}_Q$ denotes $\abs{Q}^{-1}\int_Q w$, the average of $w$ over $Q$, and $p'$ is the conjugate exponent to $p$, that is $p':=p/(p-1)$.

\section{Extension to off-diagonal and limited range extrapolation of compactness}\label{sec:extension}

In this section, we state extrapolation of compactness theorems for operators that are bounded either from $L^p$ to $L^q$, for possibly different exponents $1<p\leq q<\infty$ or on $L^p$, for a limited range of the exponent $p$. For these type of operators the following versions of Rubio de Francia's extrapolation theorems are available: 

\begin{theorem}[\cite{HMS}, Harboure--Mac\'ias--Segovia]\label{thm:Off-dig.extrp.}
Let $T$ be a linear operator defined and bounded from $L^{p_1}(\tilde w^{p_1})$ to $L^{q_1}(\tilde w^{q_1})$ for {\bf some} $1<p_1\leq q_1<\infty$ and {\bf all} $\tilde w\in A_{p_1,q_1}(\R^d)$.
Then $T$ is also defined and bounded from $L^p(w^p)$ to $L^q(w^q)$ for all $1<p\leq q<\infty$ such that $\frac{1}{p}-\frac{1}{q}=\frac{1}{p_1}-\frac{1}{q_1}$ and all $w\in A_{p,q}(\R^d)$.
\end{theorem}

This applies to the study of the fractional integral operators, also known as the {\em Riesz potentials} (see Section \ref{comm. fr. int. op.}). A version of Theorem \ref{thm:Off-dig.extrp.}, with sharp constants is due to Lacey--Moen--Per\'ez--Torres \cite{Lacey2010}. A more general version, with sharp constants and including values of $0<q<p$, was given by Duoandikoetxea \cite{D2011}.

\begin{theorem}[\cite{AM}, Theorem 4.9 of Auscher--Martell]\label{thm:limited range extrp.}
Let $1\leq p_{-}<p_{+}<\infty$, and let $T$ be a linear operator simultaneously defined and bounded on $L^{p_1}(\tilde w)$ for {\bf some} $1\leq p_{-}\leq p_1\leq p_{+}<\infty$ and {\bf all} $\tilde w\in A_{p_1/p_{-}}(\R^d)\cap RH_{(p_{+}/p_1)'}(\R^d)$. Then $T$ is also defined and bounded on $L^p(w)$ for all $p\in(p_{-},p_{+})$ and all $w\in A_{p/p_{-}}(\R^d)\cap RH_{(p_{+}/p)'}(\R^d)$.
\end{theorem}

See also \cite{CUMP:book} where these extrapolation theorems and some others are discussed. In \cite{CUMP:book}, Theorems \ref{thm:Off-dig.extrp.} and \ref{thm:limited range extrp.} are stated in terms of non-negative, measurable pairs of functions $(f,g)$. The reason is that one does not need to work with specific operators since nothing about the operators themselves is used (like linearity or sublinearity) and they play no role. However, we work with linear operators since an abstract compactness result that we will use in order to prove Theorems \ref{thm:Off-dig.extrp.compact} and \ref{thm:limited range extrp.compact} below holds for linear operators (see Theorem \ref{thm:CwKa} of Cwikel--Kalton).

As we will see below, we extend Theorem \ref{thm:RdFcompact} about the extrapolation of compactness to the setting of Theorems \ref{thm:Off-dig.extrp.} and \ref{thm:limited range extrp.}:

\begin{theorem}\label{thm:Off-dig.extrp.compact}
In the setting of Theorem \ref{thm:Off-dig.extrp.}, suppose in addition that $T$ is compact from $L^{p_1}(w_1^{p_1})$ to $L^{q_1}(w_1^{q_1})$ for {\bf some} $w_1\in A_{p_1,q_1}(\R^d)$. Then $T$ is compact from $L^p(w^p)$ to $L^q(w^q)$ for all $1<p\leq q<\infty$ such that
$\frac{1}{p}-\frac{1}{q}=\frac{1}{p_1}-\frac{1}{q_1}$ and all $w\in A_{p,q}(\R^d)$.
\end{theorem}

\begin{theorem}\label{thm:limited range extrp.compact}
In the setting of Theorem \ref{thm:limited range extrp.}, suppose in addition that $T$ is compact on $L^{p_1}(w_1)$ for {\bf some} $w_1\in A_{p_1/p_{-}}(\R^d)\cap RH_{(p_{+}/p_1)'}(\R^d)$. Then $T$ is compact on $L^{p}(w)$ for all $p\in(p_{-},p_{+})$ and all $w\in A_{p/p_{-}}(\R^d)\cap RH_{(p_{+}/p)'}(\R^d)$.
\end{theorem}

\begin{remark}
Theorems \ref{thm:limited range extrp.} and \ref{thm:limited range extrp.compact} remain true in the case $p_{+}=\infty$, provided that $p_1<\infty$ in Theorem \ref{thm:limited range extrp.compact}. In this case they reduce to Theorems \ref{thm:RdF} and \ref{thm:RdFcompact} and thus the reverse H\"older conditions on $\tilde w, w, w_1$ are vacuous.
\end{remark}

\section{Auxiliary results; proofs of Theorems \ref{thm:RdFcompact}, \ref{thm:Off-dig.extrp.compact} and \ref{thm:limited range extrp.compact} }\label{sec:main results}

We collect the results from which the proofs of Theorems \ref{thm:RdFcompact}, \ref{thm:Off-dig.extrp.compact} and \ref{thm:limited range extrp.compact} follow. Our main abstract tool is the following theorem of Cwikel--Kalton \cite{CwKa}:

\begin{theorem}[\cite{CwKa}]\label{thm:CwKa}
Let $(X_0,X_1)$ and $(Y_0,Y_1)$ be Banach couples and let $T$ be a linear operator such that $T:X_0+X_1\to Y_0+Y_1$ and $T:X_j\to Y_j$ boundedly for $j=0,1$.
Suppose moreover that $T:X_1\to Y_1$ is compact.
Let $[\ ,\ ]_\theta$ be the complex interpolation functor of Calder\'on.
Then also $T:[X_0,X_1]_\theta\to[Y_0,Y_1]_\theta$ is compact for $\theta\in(0,1)$ under {\bf any} of the following four side conditions:
\
\begin{enumerate}
  \item\label{it:UMD} $X_1$ has the UMD (unconditional martingale differences) property,
  \item\label{it:Xinterm} $X_1$ is reflexive, and $X_1=[X_0,E]_\alpha$ for some Banach space $E$ and $\alpha\in(0,1)$,
  \item\label{it:Yinterm} $Y_1=[Y_0,F]_\beta$ for some Banach space $F$ and $\beta\in(0,1)$,
  \item\label{it:lattice} $X_0$ and $X_1$ are both complexified Banach lattices of measurable functions on a common measure space.
\end{enumerate}
\end{theorem}

(We have swapped the roles of the indices $0$ and $1$ in comparison to \cite{CwKa}. For the UMD property, see Chapter 4 of \cite{HNVW1}.)
Interestingly, the question whether Theorem \ref{thm:CwKa} would remain valid {\em without any} side conditions whatsoever seems to remain open; see \cite{CwRo} for a relatively recent discussion. This is not a major concern for the present needs, as we will only use Theorem \ref{thm:CwKa} in the following special settings: (Indeed, for these present needs in weighted $L^p$ spaces, we could have replaced the use of Theorem 3.1 by a much older result of \cite[Sec. 10.4]{C}, which gives the same conclusion under yet another side condition about the existence of certain finite dimensional approximate identities in the spaces under investigation. This was kindly pointed to us by Prof. Cwikel.)

\begin{proposition}\label{prop:LpvInterm}
Let $\lambda\in[1,\infty)$, let $q,q_1\in (\lambda,\infty)$ and $v\in A_{q/\lambda}(\R^d)$, $v_1\in A_{q_1/\lambda}(\R^d)$. Then
\begin{equation*}
  L^q(v)=[L^{q_0}(v_0),L^{q_1}(v_1)]_\gamma
\end{equation*}
for some $q_0\in(\lambda,\infty)$, $v_0\in A_{q_0/\lambda}(\R^d)$, and $\gamma\in(0,1)$.
\end{proposition}

\begin{proposition}\label{prop:offdiagonal}
Suppose that $1<p\leq q<\infty$, $1<p_1\leq q_1<\infty$ and $v\in A_{p,q}(\R^d)$, $v_1\in A_{p_1,q_1}(\R^d)$. Then
\begin{equation*}
  [L^{p_0}({v_0}^{p_0}),L^{p_1}({v_1}^{p_1})]_\gamma=L^p(v^p)\,\,\,\text{and}\,\,\,[L^{q_0}({v_0}^{q_0}),L^{q_1}({v_1}^{q_1})]_\gamma=L^q(v^q)
\end{equation*}
for some $1<p_0\leq q_0<\infty$, $v_0\in A_{p_0,q_0}(\R^d)$, and $\gamma\in(0,1)$. Moreover, if $\frac{1}{p}-\frac{1}{q}=\frac{1}{p_1}-\frac{1}{q_1}$ we can choose $p_0,q_0$ in such a way that $\frac{1}{p}-\frac{1}{q}=\frac{1}{p_0}-\frac{1}{q_0}$.
\end{proposition}

\begin{proposition}\label{prop:limitedrange}
Suppose that $1\leq p_{-}<p_{+}<\infty$, $q_1\in[p_{-},p_{+}]$, $q\in(p_{-},p_{+})$ and
\begin{equation*}
v\in A_{q/p_{-}}(\R^d)\cap RH_{(p_{+}/q)'}(\R^d),\qquad v_1\in A_{q_1/p_{-}}(\R^d)\cap RH_{(p_{+}/q_1)'}(\R^d). 
\end{equation*}
Then
\begin{equation*}
  [L^{q_0}(v_0),L^{q_1}(v_1)]_{\gamma}=L^q(v)
\end{equation*}
for some $q_0\in(p_{-},p_{+})$, $v_0\in A_{q_0/p_{-}}(\R^d)\cap RH_{(p_{+}/q_0)'}(\R^d)$, and $\gamma\in(0,1)$.
\end{proposition}

We postpone the proofs of Propositions \ref{prop:LpvInterm}, \ref{prop:offdiagonal} and \ref{prop:limitedrange} to the following section. The verifications of these propositions are the only components of the proofs of Theorems \ref{thm:RdFcompact}, \ref{thm:Off-dig.extrp.compact} and \ref{thm:limited range extrp.compact} that require actual computations, rather than just a soft application of known results. 

\begin{lemma}\label{lem:lemOk}
If $p_j\in[1,\infty)$ and $w_j$ are weights, then the spaces $X_j=L^{p_j}(w_j)$ satisfy the condition \eqref{it:lattice} of Theorem \ref{thm:CwKa}.
\end{lemma}

\begin{proof}
It is immediate that both $X_j=L^{p_j}(w_j)$, $j=0,1$, are complexified Banach lattices of measurable functions on the common measure space $\R^d$.
\end{proof}

\begin{remark}
If $p_j,q_j\in(1,\infty)$ and $u_j,v_j$ are weights,
then the condition \eqref{it:UMD} of Theorem \ref{thm:CwKa} is also satisfied  by the spaces $X_j=L^{p_j}(u_j)$ and $Y_j=L^{q_j}(v_j)$. Indeed, it is well known (see e.g. Proposition 4.2.15 in \cite{HNVW1}) that all $L^p(\mu)$ spaces with $p\in(1,\infty)$ have the UMD property for any measure $\mu$, so in particular for the weighted Lebesgue measure.  Moreover, in all our concrete applications where the weights belong to appropriate Muckenhoupt classes we could also verify conditions \eqref{it:Xinterm} and \eqref{it:Yinterm} but we do not give the details here. For example, if $u_j\in A_{p_j/\lambda}(\R^d)$ then by Proposition \ref{prop:LpvInterm} we could deduce these conditions. Similar considerations could be done in the settings of Propositions \ref{prop:offdiagonal} and \ref{prop:limitedrange}.

For applications of Theorem \ref{thm:CwKa} to these concrete spaces, this is of course more than sufficient. We would only need one of the four side conditions, but in fact we have them all. Indeed, as already mentioned, we could have also verified the much older condition from Section 10.4 in \cite{C} instead of those of Theorem \ref{thm:CwKa}.  
\end{remark}

We can now give the proof of our main results:

\begin{proof}[Proof of Theorem \ref{thm:RdFcompact}]
Recall that the assumptions, and hence the conclusions, of Theorem \ref{thm:RdF} are in force. In particular, $T$ is a bounded linear operator on $L^p(w)$ for all $p\in(\lambda,\infty)$ and all $w\in A_{p/\lambda}(\R^d)$. In addition, it is assumed that $T$ is a compact operator on $L^{p_1}(w_1)$ for some $p_1\in(\lambda,\infty)$ and some $w_1\in A_{p_1/\lambda}(\R^d)$. We need to prove that $T$ is actually compact on $L^p(w)$ for all $p\in(\lambda,\infty)$ and all $w\in A_{p/\lambda}(\R^d)$. Now, fix some $p\in(\lambda,\infty)$ and $w\in A_{p/\lambda}(\R^d)$. By Proposition \ref{prop:LpvInterm}, we have
\begin{equation*}
  L^p(w)=[L^{p_0}(w_0),L^{p_1}(w_1)]_\theta
\end{equation*}
for some $p_0\in(\lambda,\infty)$, some $w_0\in A_{p_0/\lambda}(\R^d)$, and some $\theta\in(0,1)$. Writing $X_j=Y_j=L^{p_j}(w_j)$, we know that $T:X_0+X_1\to Y_0+Y_1$, that $T:X_j\to Y_j$ is bounded (since $T$ is bounded on all $L^q(w)$ with $q\in(\lambda,\infty)$ and $w\in A_{q/\lambda}(\R^d)$ by Theorem \ref{thm:RdF}), and that $T:X_1\to Y_1$ is compact (since this was assumed). By Lemma \ref{lem:lemOk}, the last condition \eqref{it:lattice} of Theorem \ref{thm:CwKa} is also satisfied by these spaces $X_j=L^{p_j}(w_j)$. By Theorem \ref{thm:CwKa}, it follows that $T$ is also compact on $[X_0,X_1]_\theta=[Y_0,Y_1]_\theta=L^p(w)$.
\end{proof}

\begin{proof}[Proof of Theorem \ref{thm:Off-dig.extrp.compact}]
Recall that the assumptions, and hence the conclusions, of Theorem \ref{thm:Off-dig.extrp.} are in force. In particular, $T:L^p(w^p)\to L^q(w^q)$ is a bounded linear operator for all $1<p\leq q<\infty$ such that $\frac{1}{p}-\frac{1}{q}=\frac{1}{p_1}-\frac{1}{q_1}$ and all $w\in A_{p,q}(\R^d)$. In addition, it is assumed that $T:L^{p_1}(w_1^{p_1})\to L^{q_1}(w_1^{q_1})$ is a compact operator for some $1<p_1\leq q_1<\infty$ and some $w_1\in A_{p_1,q_1}(\R^d)$. We need to prove that $T:L^p(w^p)\to L^q(w^q)$ is actually compact for all $1<p\leq q<\infty$ such that $\frac{1}{p}-\frac{1}{q}=\frac{1}{p_1}-\frac{1}{q_1}$ and all $w\in A_{p,q}(\R^d)$. Now, fix some $1<p\leq q<\infty$ and $w\in A_{p,q}(\R^d)$. By Proposition \ref{prop:offdiagonal}, we have
\begin{equation*}
  L^p(w^p)=[L^{p_0}({w_0}^{p_0}),L^{p_1}({w_1}^{p_1})]_\theta\,\,\,\text{and}\,\,\,L^q(w^q)=[L^{q_0}({w_0}^{q_0}),L^{q_1}({w_1}^{q_1})]_\theta
\end{equation*}
for some $1<p_0\leq q_0<\infty$, some $w_0\in A_{p_0,q_0}(\R^d)$, some $\theta\in(0,1)$ and $\frac{1}{p}-\frac{1}{q}=\frac{1}{p_0}-\frac{1}{q_0}$. Writing $X_j=L^{p_j}(w_j^{p_j})$ and $Y_j=L^{q_j}(w_j^{q_j})$, we know that $T:X_0+X_1\to Y_0+Y_1$ and $T:X_j\to Y_j$ is bounded ($T:L^{\tilde p}(w^{\tilde p})\to L^{\tilde q}(w^{\tilde q})$ is a bounded linear operator for all $1<\tilde p\leq\tilde q<\infty$ such that $\frac{1}{\tilde p}-\frac{1}{\tilde q}=\frac{1}{\tilde p_1}-\frac{1}{\tilde q_1}$ and all $w\in A_{\tilde p,\tilde q}(\R^d)$ by Theorem \ref{thm:Off-dig.extrp.}), and that $T:X_1\to Y_1$ is compact (since this was assumed). By Lemma \ref{lem:lemOk}, the last condition \eqref{it:lattice} of Theorem \ref{thm:CwKa} is also satisfied by these spaces $X_j=L^{p_j}(w_j^{p_j})$.
By Theorem \ref{thm:CwKa}, it follows that $T:L^p(w^p)=[X_0,X_1]_\theta\to L^q(w^q)=[Y_0,Y_1]_\theta$ is also compact.
\end{proof}

\begin{proof}[Proof of Theorem \ref{thm:limited range extrp.compact}]
Recall that the assumptions, and hence the conclusions, of Theorem \ref{thm:limited range extrp.} are in force. In particular, $T$ is a bounded linear operator on $L^p(w)$ for all $p\in(p_{-},p_{+})$ and all $w\in A_{p/p_{-}}(\R^d)\cap RH_{(p_{+}/p)'}(\R^d)$. In addition, it is assumed that $T$ is a compact operator on $L^{p_1}(w_1)$ for some $p_1\in[p_{-},p_{+}]$ and some $w_1\in A_{p_1/p_{-}}(\R^d)\cap RH_{(p_{+}/p_1)'}(\R^d)$. We need to prove that $T$ is actually compact on $L^p(w)$ for all $p\in(p_{-},p_{+})$ and all $w\in A_{p/p_{-}}(\R^d)\cap RH_{(p_{+}/p)'}(\R^d)$. Now, fix some $p\in(p_{-},p_{+})$ and $w\in A_{p/p_{-}}(\R^d)\cap RH_{(p_{+}/p)'}(\R^d)$. By Proposition \ref{prop:limitedrange}, we have
\begin{equation*}
  L^p(w)=[L^{p_0}(w_0),L^{p_1}(w_1)]_\theta
\end{equation*}
for some $p_0\in(p_{-},p_{+})$, some $w_0\in A_{p_0/p_{-}}(\R^d)\cap RH_{(p_{+}/p_0)'}(\R^d)$ and some $\theta\in(0,1)$. Writing $X_j=Y_j=L^{p_j}(w_j)$, we know that $T:X_0+X_1\to Y_0+Y_1$, that $T:X_j\to Y_j$ is bounded (since $T$ is bounded on all $L^q(w)$ with $q\in(p_{-},p_{+})\cup\{p_1\}$ and $w\in A_{q/p_{-}}(\R^d)\cap RH_{(p_{+}/q)'}(\R^d)$ by the assumptions and the conclusion of Theorem \ref{thm:limited range extrp.}), and also that $T:X_1\to Y_1$ is compact (since this was assumed). By Lemma \ref{lem:lemOk}, the last condition \eqref{it:lattice} of Theorem \ref{thm:CwKa} is also satisfied by these spaces $X_j=L^{p_j}(w_j)$. By Theorem \ref{thm:CwKa}, it follows that $T$ is also compact on $[X_0,X_1]_\theta=[Y_0,Y_1]_\theta=L^p(w)$.
\end{proof}

\section{The Proofs of Propositions \ref{prop:LpvInterm}, \ref{prop:offdiagonal} and \ref{prop:limitedrange}}\label{sec:prop}

To complete the proofs of Theorems \ref{thm:RdFcompact}, \ref{thm:Off-dig.extrp.compact} and \ref{thm:limited range extrp.compact}, it remains to verify Propositions \ref{prop:LpvInterm}, \ref{prop:offdiagonal} and \ref{prop:limitedrange}. We quote two more classical results:

\begin{proposition}[\cite{RDF1985, Gehring, JN1991}]\label{prop:weights} The following statements hold:

\begin{enumerate}
\item \label{weights prop. 1} \ (Theorem 1.14 in \cite{RDF1985}) If $1<p<\infty$, we have $w\in A_p(\R^d)$ if and only if $w^{1-p'}\in
A_{p'}(\R^d)$.
\item \label{weights prop. 2} \ (Theorem 2.6 in \cite{RDF1985}) If $w\in A_p(\R^d)$, $1<p<\infty$, then there exists $1<q<p$ such
that $w\in A_q(\R^d)$.
\item \label{weights prop. 3} \ (Lemma 3 in \cite{Gehring}) If $w\in RH_q(\R^d)$, $1<q<\infty$, then there exists $q<p<\infty$ such
that $w\in RH_p(\R^d)$.
\item \label{weights prop. 4} \ If $w\in A_{p,q}(\R^d)$, for $1<p\leq q<\infty$, then $w^q\in A_{1+q/p'}(\R^d)$ and $w^{-p'}\in A_{1+p'/q}(\R^d)$, where $\frac{1}{p}+\frac{1}{p'}=1$.
\item \label{weights prop. 5} \ (Statement (P6) in \cite{JN1991}) If $1<q,s<\infty$, then $w\in A_q(\R^d) \cap RH_s(\R^d)$ if and only if $w^{s}\in A_{s\,(q-1)+1}(\R^d)$.
\end{enumerate}
\end{proposition}

\begin{proof}
We only prove property \eqref{weights prop. 4}. Notice that $w\in A_{p,q}(\R^d)\Longleftrightarrow w^q\in A_r(\R^d)$, with $[w]_{A_{p,q}}=[w^q]_{A_r}$, where 
\begin{equation*}
  r:=1+q/p'.
\end{equation*}
The proof of $w^{-p'}\in A_{1+p'/q}(\R^d)$ follows in a similar fashion.
\end{proof}

\begin{theorem}[\cite{BL,SW:58}]\label{thm:SW}
If $q_0,q_1\in[1,\infty)$ and $w_0,w_1$ are two weights, then for all $\theta\in(0,1)$ we have
\begin{equation*}
  [L^{q_0}(w_0),L^{q_1}(w_1)]_\theta=L^q(w),
\end{equation*}
where
\begin{equation}\label{eq:convexity}
  \frac{1}{q}=\frac{1-\theta}{q_0}+\frac{\theta}{q_1}\qquad\text{and}\qquad
  w^{\frac{1}{q}}=w_0^{\frac{1-\theta}{q_0}}w_1^{\frac{\theta}{q_1}}.
\end{equation}
\end{theorem}

As stated, this can be found in Theorem 5.5.3 of \cite{BL}, but it is essentially a reformulation of an old theorem of Stein--Weiss \cite{SW:58} that predates the general interpolation theory.
In order to connect Theorem \ref{thm:SW} with $A_p(\R^d)$,  $A_{p,q}(\R^d)$ and $A_{q/p_{-}}(\R^d)\cap RH_{(p_{+}/q)'}(\R^d)$ weights, we need:

\begin{lemma}\label{lem:main}
Let $\lambda\in[1,\infty)$, $q_1,q\in(\lambda,\infty)$, $w_1\in A_{q_1/\lambda}(\R^d)$, $w\in A_{q/\lambda}(\R^d)$. Then there exist $q_0\in(\lambda,\infty)$, $w_0\in A_{q_0/\lambda}(\R^d)$, and $\theta\in(0,1)$ such that \eqref{eq:convexity} holds.
\end{lemma}

\begin{proof}
Note that the choice of $\theta\in(0,1)$ determines both
\begin{equation*}
  q_0=q_0(\theta)=\frac{1-\theta}{\frac{1}{q}-\frac{\theta}{q_1}}\qquad\text{and}\qquad
  w_0=w_0(\theta)=w^{\frac{q_0}{q(1-\theta)}}w_1^{-\frac{q_0\cdot\theta}{q_1(1-\theta)}},
\end{equation*}
so it remains to check that we can choose $\theta\in(0,1)$ so that $q_0\in(\lambda,\infty)$ and $w_0\in A_{q_0/\lambda}(\R^d)$. Since $q_0(0)=q\in(\lambda,\infty)$, the first condition is obvious for small enough $\theta>0$ by continuity. To simplify writing, we denote $p_i:=q_i/\lambda$ for $i=0,1$ and $p:=q/\lambda$, and observe that these satisfy the same relations
\begin{equation*}
  p_0=p_0(\theta)=\frac{q_0(\theta)}{\lambda}=\frac{1-\theta}{\frac{\lambda}{q}-\frac{\theta\lambda}{q_1}}
  =\frac{1-\theta}{\frac{1}{p}-\frac{\theta}{p_1}}
\end{equation*}
and $p_0(0)=p$ as the original exponent $q_i$ and $q$.

We need to check that $w_0\in A_{p_0}(\R^d)$, so we consider a cube $Q$ and write
\begin{equation*}
\begin{split}
  &\ave{w_0}_Q\ave{w_0^{-\frac{1}{p_0-1}}}_Q^{p_0-1}
  =\ave{w^{\frac{p_0}{p(1-\theta)}}w_1^{-\frac{p_0\cdot\theta}{p_1(1-\theta)}}}_Q
    \ave{w^{-\frac{p_0'}{p(1-\theta)}}w_1^{\frac{p_0'\cdot\theta}{p_1(1-\theta)}}}_Q^{p_0-1} \\
  &=\ave{w^{\frac{p_0}{p(1-\theta)}}(w_1^{-\frac{1}{p_1-1}})^{\frac{p_0\cdot\theta}{p_1'(1-\theta)}}}_Q
    \ave{(w^{-\frac{1}{p-1}})^{\frac{p_0'}{p'(1-\theta)}}w_1^{\frac{p_0'\cdot\theta}{p_1(1-\theta)}}}_Q^{p_0-1},
\end{split}
\end{equation*}
where $u':=u/(u-1)$ denotes the conjugate exponent of $u\in\{p,p_0,p_1\}$.

In the first average, we use H\"older's inequality with exponents $1+\eps^{\pm 1}$, and in the second with exponents $1+\delta^{\pm 1}$ to get
\begin{equation}\label{eq:beforeRHImain}
\begin{split}
  &\leq \ave{w^{\frac{p_0(1+\eps)}{p(1-\theta)}}}_Q^{\frac{1}{1+\eps}}  \ave{(w_1^{-\frac{1}{p_1-1}})^{\frac{p_0\theta(1+\eps)}{p_1'\eps(1-\theta)}}}_Q^{\frac{\eps}{1+\eps}}
  \ave{(w^{-\frac{1}{p-1}})^{\frac{p_0'(1+\delta)}{p'(1-\theta)}} }_Q^{\frac{p_0-1}{1+\delta}}  \\ 
  &\qquad\times\ave{ w_1^{\frac{p_0'\theta(1+\delta)}{p_1\delta(1-\theta)}}}_Q^{\frac{\delta(p_0-1)}{1+\delta}}  \\
  &=\ave{w^{r(\theta)}}_Q^{\frac{1}{1+\eps}}\ave{(w_1^{-\frac{1}{p_1-1}})^{s(\theta)}}_Q^{\frac{\eps}{1+\eps}}
  \ave{(w^{-\frac{1}{p-1}})^{t(\theta)} }_Q^{\frac{p_0-1}{1+\delta}}  \\
  &\qquad\times\ave{ w_1^{u(\theta)}}_Q^{\frac{\delta(p_0-1)}{1+\delta}},
\end{split}
\end{equation}
where
\begin{equation*}
  r(\theta):=\frac{p_0(\theta)(1+\eps)}{p(1-\theta)},\qquad
  s(\theta):=\frac{\theta p_0(\theta)(1+\eps)}{p_1'\eps(1-\theta)}
\end{equation*}
and
\begin{equation*}
  t(\theta):=\frac{p_0'(\theta)(1+\delta)}{p'(1-\theta)},\qquad
  u(\theta):=\frac{\theta p_0'(\theta)(1+\delta)}{p_1\delta(1-\theta)}.
\end{equation*}

Now, we choose $\eps=\frac{\theta p}{p_1'}$ and $\delta=\frac{\theta p'}{p_1}$ in such a way that
\begin{equation*}
  r(\theta)=s(\theta)=\frac{p_0(\theta)(p_1'+\theta p)}{pp_1'(1-\theta)},
\end{equation*}
and
\begin{equation*}
  t(\theta)=u(\theta)=\frac{p_0(\theta)'(p_1+\theta p')}{p'p_1(1-\theta)}. 
\end{equation*}

The strategy to proceed is to use the reverse H\"older inequality for $A_u(\R^d)$ weights due to Coifman--Fefferman \cite{CF}, which says that each $v\in A_u(\R^d)$ satisfies
\begin{equation}\label{eq:RHI}
  \ave{v^t}_Q^{1/t}\lesssim \ave{v}_Q
\end{equation}
for all $t\leq 1+\eta$, for some $\eta>0$ depending only on $[v]_{A_u}$. (For a sharp quantitative version, see Theorem 2.3 in \cite{HytPer}.)

Recalling that $p_0(0)=p$, we see that $r(0)=s(0)=1$. By continuity, given any $\eta>0$, we find that 
\begin{equation*}
  \max(r(\theta),s(\theta))\leq 1+\eta\,\,\,\text{for all small enough}\,\,\,\theta>0.
\end{equation*}
By property \eqref{weights prop. 1} of Proposition \ref{prop:weights} each of the four functions $w\in A_p(\R^d)$, $w^{-\frac{1}{p-1}}\in A_{p'}(\R^d)$, $w_1\in A_{p_1}(\R^d)$ and $w_1^{-\frac{1}{p_1-1}}\in A_{p_1'}(\R^d)$ satisfies the reverse H\"older inequality \eqref{eq:RHI} for all $t\leq 1+\eta$ and for some $\eta>0$. Thus, for all small enough $\theta>0$, we have
\begin{equation*}
\begin{split}
  \eqref{eq:beforeRHImain}
  &\lesssim \ave{w}_Q^{\frac{p_0}{p(1-\theta)}} \ave{w_1^{-\frac{1}{p_1-1}}}_Q^{\frac{\theta p_0}{p_1'(1-\theta)}}
  \ave{w^{-\frac{1}{p-1}} }_Q^{\frac{p_0'(p_0-1)}{p'(1-\theta)}}
  \ave{ w_1 }_Q^{\frac{\theta p_0'(p_0-1)}{p_1(1-\theta)}}  \\
  &=\ave{w}_Q^{\frac{p_0(\theta) }{p(1-\theta)}} \ave{w_1^{-\frac{1}{p_1-1}}}_Q^{\frac{\theta p_0(\theta)}{p_1'(1-\theta)}}\ave{w^{-\frac{1}{p-1}} }_Q^{\frac{p_0(\theta)}{p'(1-\theta)}}\ave{ w_1 }_Q^{\frac{\theta p_0(\theta)}{p_1(1-\theta)}}  \\
  &=(\ave{w}_Q\ave{w^{-\frac{1}{p-1}} }_Q^{p-1})^{\frac{p_0(\theta) }{p(1-\theta)}}(\ave{ w_1 }_Q\ave{w_1^{-\frac{1}{p_1-1}}}_Q^{p_1-1})^{\frac{\theta p_0(\theta)}{p_1(1-\theta)}}  \\
 & \leq [w]_{A_p}^{\frac{p_1}{p_1-\theta p}}[w_1]_{A_{p_1}}^{\frac{\theta p}{p_1-\theta p}}.
\end{split}
\end{equation*}
In combination with the lines preceding \eqref{eq:beforeRHImain}, we have shown that
\begin{equation*}
  [w_0]_{A_{p_0}}\lesssim  [w]_{A_p}^{\frac{p_1}{p_1-\theta p}}[w_1]_{A_{p_1}}^{\frac{\theta p}{p_1-\theta p}}<\infty,
\end{equation*}
provided that $\theta>0$ is small enough. This concludes the proof.
\end{proof}

\begin{lemma}\label{lem:main1}
Let $1<p_1\leq q_1<\infty$, $1<p\leq q<\infty$, $w_1\in A_{p_1,q_1}(\R^d)$, $w\in A_{p,q}(\R^d)$. Then there exist $1<p_0\leq q_0<\infty$, $w_0\in A_{p_0,q_0}(\R^d)$, and $\theta\in(0,1)$ such that the conclusion of Theorem \ref{thm:SW} holds, i.e.,
\begin{equation*}
  [L^{p_0}({w_0}^{p_0}),L^{p_1}({w_1}^{p_1})]_\theta=L^p(w^p),\qquad[L^{q_0}({w_0}^{q_0}),L^{q_1}({w_1}^{q_1})]_\theta=L^q(w^q)
\end{equation*}
where
\begin{equation*}
  \frac{1}{p}=\frac{1-\theta}{p_0}+\frac{\theta}{p_1},\qquad\frac{1}{q}=\frac{1-\theta}{q_0}+\frac{\theta}{q_1},\qquad
  w=w_0^{1-\theta}w_1^{\theta}.
\end{equation*}
\end{lemma}

\begin{proof}
Note that the choice of $\theta\in(0,1)$ determines 
\begin{equation*}
  p_0=p_0(\theta)=\frac{1-\theta}{\frac{1}{p}-\frac{\theta}{p_1}},\quad
  q_0=q_0(\theta)=\frac{1-\theta}{\frac{1}{q}-\frac{\theta}{q_1}},\quad
  w_0=w_0(\theta)=w^{\frac{1}{1-\theta}}w_1^{-\frac{\theta}{1-\theta}},
\end{equation*}
so it remains to check that we can choose $\theta\in(0,1)$ so that $1<p_0\leq q_0<\infty$ and $w_0\in A_{p_0,q_0}(\R^d)$. Since $1<p_0(0)=p\leq q=q_0(0)<\infty$, the first condition is obvious for small enough $\theta>0$ by continuity. 

We need to check that $w_0\in A_{p_0,q_0}(\R^d)$, so we consider a cube $Q$ and write
\begin{equation*}
  \ave{w_0^{q_0}}_Q^{\frac{1}{q_0}}\ave{w_0^{-p_0'}}_Q^{\frac{1}{p_0'}}
  =\ave{w^{\frac{q_0}{1-\theta}}w_1^{-\frac{q_0\cdot\theta}{1-\theta}}}_Q^{\frac{1}{q_0}}
  \ave{w^{-\frac{p_0'}{1-\theta}}w_1^{\frac{p_0'\cdot\theta}{1-\theta}}}_Q^{\frac{1}{p_0'}},
\end{equation*}
where $p_0':=p_0/(p_0-1)$ denotes the conjugate exponent of $p_0$.

In the first average, we use H\"older's inequality with exponents $1+\eps^{\pm 1}$, and in the second with exponents $1+\delta^{\pm 1}$ to get
\begin{equation}\label{eq:beforeRHI1}
\begin{split}
  &\leq \ave{w^{\frac{q_0(1+\eps)}{1-\theta}}}_Q^{\frac{1}{q_0(1+\eps)}} \ave{w_1^{-\frac{q_0\theta(1+\eps)}{\eps(1-\theta)}}}_Q^{\frac{\eps}{q_0(1+\eps)}}  \\ &\qquad\times\ave{w^{-\frac{p_0'(1+\delta)}{1-\theta}}}_Q^{\frac{1}{p_0'(1+\delta)}} \ave{w_1^{\frac{p_0'\theta(1+\delta)}{\delta(1-\theta)}}}_Q^{\frac{\delta}{p_0'(1+\delta)}}  \\
  &=\ave{(w^{q})^{r(\theta)}}_Q^{\frac{1}{q_0(1+\eps)}} \ave{(w_1^{-p_1'})^{s(\theta)}}_Q^{\frac{\eps}{q_0(1+\eps)}}  \\
  &\qquad\times\ave{(w^{-p'})^{t(\theta)}}_Q^{\frac{1}{p_0'(1+\delta)}} \ave{ (w_1^{q_1})^{u(\theta)}}_Q^{\frac{\delta}{p_0'(1+\delta)}},
\end{split}
\end{equation}
where
\begin{equation*}
  r(\theta):=\frac{q_0(\theta)(1+\eps)}{q(1-\theta)},\qquad
  s(\theta):=\frac{\theta q_0(\theta)(1+\eps)}{p_1'\eps(1-\theta)}
\end{equation*}
and
\begin{equation*}
  t(\theta):=\frac{p_0'(\theta)(1+\delta)}{p'(1-\theta)},\qquad
  u(\theta):=\frac{\theta p_0(\theta)'(1+\delta)}{q_1\delta(1-\theta)}.
\end{equation*}

Now, we choose $\eps=\frac{\theta q}{p_1'}$ and $\delta=\frac{\theta p'}{q_1}$ in such a way that
\begin{equation*}
  r(\theta)= s(\theta)=\frac{q_0(\theta)(p_1'+\theta q)}{qp_1'(1-\theta)},
\end{equation*}
and
\begin{equation*}
  t(\theta)=u(\theta)=\frac{p_0'(\theta)(q_1+\theta p')}{q_1 p'(1-\theta)}.
\end{equation*}

The strategy to proceed is the same as in the proof of Lemma \ref{lem:main}. In particular, we use the reverse H\"older inequality (\ref{eq:RHI}) for $A_v(\R^d)$ weights.
Recalling that $p_0(0)=p$ and $q_0(0)=q$, we see that $r(0)=t(0)=1$. By continuity, given any $\eta>0$, we find that 
\begin{equation*}
\max(r(\theta),t(\theta))\leq 1+\eta\,\,\,\text{for all small enough}\,\,\,\theta>0.
\end{equation*}
By property \eqref{weights prop. 4} of Proposition \ref{prop:weights} each of the four functions $w^q\in A_{1+\frac{q}{p'}}(\R^d)$, $w^{-p'}\in A_{1+p'/q}(\R^d)$, $w_1^{q_1}\in A_{1+q_1/p_1'}(\R^d)$ and $w_1^{-p_1'}\in A_{1+p_1'/q_1}(\R^d)$ satisfies the reverse H\"older inequality (\ref{eq:RHI}) for all $t\leq 1+\eta$ and for some $\eta>0$. Thus, for all small enough $\theta>0$, we have
\begin{equation*}
\begin{split}
  \eqref{eq:beforeRHI1}
  &\lesssim  \ave{w^{q}}_Q^{\frac{1}{q(1-\theta)}} \ave{w_1^{-p_1'}}_Q^{\frac{\theta}{p_1'(1-\theta)}}  \\
  &\qquad\times\ave{w^{-p'}}_Q^{\frac{1}{p'(1-\theta)}} \ave{w_1^{q_1}}_Q^{\frac{\theta}{q_1(1-\theta)}}  \\
  &=\bigg(\ave{w^q}_Q^{\frac{1}{q}} \ave{w^{-p'}}_Q^{\frac{1}{p'}}\bigg)^{\frac{1}{1-\theta}}
    \bigg(\ave{w_1^{q_1}}_Q^{\frac{1}{q_1}} \ave{w_1^{-p_1'}}_Q^{\frac{1}{p_1'}}\bigg)^{\frac{\theta}{1-\theta}}  \\
  &\leq [w]_{A_{p,q}}^{\frac{1}{1-\theta}}[w_1]_{A_{p_1,q_1}}^{\frac{\theta}{1-\theta}}.
\end{split}
\end{equation*}
In combination with the lines preceding \eqref{eq:beforeRHI1}, we have shown that
\begin{equation*}
  [w_0]_{A_{p_0,q_0}}\lesssim  [w]_{A_{p,q}}^{\frac{1}{1-\theta}}[w_1]_{A_{p_1,q_1}}^{\frac{\theta}{1-\theta}}<\infty,
\end{equation*}
provided that $\theta>0$ is small enough. This concludes the proof.
\end{proof}

\begin{lemma}\label{lem:main2}
Let $1\leq p_{-}<p_{+}<\infty$, $q_1\in[p_{-},p_{+}]$, $q\in(p_{-},p_{+})$, and 
\begin{equation*}
  w_1\in A_{q_1/p_{-}}(\R^d)\cap RH_{(p_{+}/q_1)'}(\R^d),\qquad w\in A_{q/p_{-}}(\R^d)\cap
  RH_{(p_{+}/q)'}(\R^d).
\end{equation*}
Then there exist $q_0\in(p_{-},p_{+})$, $w_0\in A_{q_0/p_{-}}(\R^d)\cap RH_{(p_{+}/q_0)'}(\R^d)$, and $\theta\in(0,1)$ such that (\ref{eq:convexity}) holds.
\end{lemma}

\begin{proof}
We prove the lemma in the following three separate cases: $q_1\in(p_{-},p_{+})$, $q_1=p_{-}$ and $q_1=p_{+}$. Let us assume that $q_1\in(p_{-},p_{+})$. By property \eqref{weights prop. 5} of Proposition \ref{prop:weights} we prove the lemma in its equivalent form:
if $v_1:=w_1^{(p_{+}/q_1)'}\in A_{s_1}(\R^d)$ and $v:=w^{(p_{+}/q)'}\in A_{s}(\R^d)$ then there exist $q_0\in(p_{-},p_{+})$, $v_0:=w_0^{(p_{+}/q_0)'}\in A_{s_0}(\R^d)$, and $\theta\in(0,1)$ such that
\begin{equation*}
  [L^{q_0}(w_0),L^{q_1}(w_1)]_\theta=L^q(w),
\end{equation*}
where
\begin{equation*}
  \frac{1}{q}=\frac{1-\theta}{q_0}+\frac{\theta}{q_1},\qquad
  w^{\frac{1}{q}}=w_0^{\frac{1-\theta}{q_0}}w_1^{\frac{\theta}{q_1}},
\end{equation*}
and
\begin{equation*}
\begin{split}
  &s_1=\bigg(\frac{p_{+}}{q_1}\bigg)'\bigg(\frac{q_1}{p_{-}}-1\bigg)+1,  \\ &s=\bigg(\frac{p_{+}}{q}\bigg)'\bigg(\frac{q}{p_{-}}-1\bigg)+1,  \\ 
  &s_0=\bigg(\frac{p_{+}}{q_0}\bigg)'\bigg(\frac{q_0}{p_{-}}-1\bigg)+1.
\end{split}
\end{equation*}

Note that the choice of $\theta\in(0,1)$ determines both 
\begin{equation*}
  q_0=q_0(\theta)=\frac{1-\theta}{\frac{1}{q}-\frac{\theta}{q_1}},\quad
  w_0=w_0(\theta)=w^{\frac{q_0}{q(1-\theta)}}w_1^{-\frac{q_0\cdot\theta}{q_1(1-\theta)}},
\end{equation*}
so it remains to check that we can choose $\theta\in(0,1)$ so that $q_0\in(p_{-},p_{+})$ and $v_0=w_0^{(p_{+}/q_0)'}\in A_{s_0}(\R^d)$, where $s_0=\big(\frac{p_{+}}{q_0}\big)'\big(\frac{q_0}{p_{-}}-1\big)+1$. Since $q_0(0)=q\in(p_{-},p_{+})$, the first condition is obvious for small enough $\theta>0$ by continuity. 

We need to check that $v_0=w_0^{(p_{+}/q_0)'}\in A_{s_0}(\R^d)$, so we consider a cube $Q$ and write
\begin{equation}\label{eq:beforerhi}
\begin{split}
  \ave{v_0}_Q\ave{v_0^{-\frac{1}{s_0-1}}}_Q&= \ave{w_0^{(p_{+}/q_0)'}}_Q\ave{w_0^{(p_{+}/q_0)'(-\frac{1}{s_0-1})}}_Q^{s_0-1}  \\
  &= \ave{w^{\frac{q_0(p_{+}/q_0)'}{q(1-\theta)}}w_1^{-\frac{q_0\cdot\theta(p_{+}/q_0)'}{q_1(1-\theta)}}}_Q  \\
  &\qquad\times\ave{w^{-\frac{q_0(p_{+}/q_0)'}{q(1-\theta)(s_0-1)}}w_1^{\frac{q_0\cdot\theta(p_{+}/q_0)'}{q_1(1-\theta)(s_0-1)}}}_Q^{s_0-1}.
\end{split}
\end{equation}

In the first average, we use H\"older's inequality with exponents $1+\eps^{\pm 1}$ and in the second with exponents $1+\delta^{\pm 1}$ to get
\begin{equation}\label{eq:beforeRHI2}
\begin{split}
  &\leq \ave{w^{\frac{q_0(p_{+}/q_0)'(1+\eps)}{q(1-\theta)}}}_Q^\frac{1}{1+\eps} \ave{w_1^{-\frac{q_0\theta(p_{+}/q_0)'(1+\eps)}{q_1\eps(1-\theta)}}}_Q^\frac{\eps}{1+\eps}   \\ 
  &\qquad\times\ave{w^{-\frac{q_0(p_{+}/q_0)'(1+\delta)}{q(1-\theta)(s_0-1)}}}_Q^{\frac{s_0-1}{1+\delta}} \ave{w_1^{\frac{q_0\theta(p_{+}/q_0)'(1+\delta)}{q_1\delta(1-\theta)(s_0-1)}}}_Q^\frac{(s_0-1)\delta}{1+\delta}  \\
  &= \ave{(w^{(p_{+}/q)'})^{\tilde r(\theta)}}_Q^\frac{1}{1+\eps} \ave{(w_1^{(p_{+}/q_1)'(-\frac{1}{s_1-1})})^{\tilde s(\theta)}}_Q^\frac{\eps}{1+\eps}  \\
  &\qquad\times\ave{(w^{(p_{+}/q)'(-\frac{1}{s-1})})^{\tilde t(\theta)}}_Q^{\frac{s_0-1}{1+\delta}} \ave{(w_1^{(p_{+}/q_1)'})^{\tilde u(\theta)}}_Q^\frac{(s_0-1)\delta}{1+\delta}  \\
  &=\ave{v^{\tilde r(\theta)}}_Q^\frac{1}{1+\eps} \ave{(v_1^{-\frac{1}{s_1-1}})^{\tilde s(\theta)}}_Q^\frac{\eps}{1+\eps}  \\
  &\qquad\times\ave{(v^{-\frac{1}{s-1}})^{\tilde t(\theta)}}_Q^{\frac{s_0-1}{1+\delta}} \ave{v_1^{\tilde u(\theta)}}_Q^\frac{(s_0-1)\delta}{1+\delta},
\end{split}
\end{equation}
where
\begin{equation*}
  \tilde r(\theta):=\frac{q_0(\theta)(p_{+}-q)(1+\eps)}{q(1-\theta)(p_{+}-q_0(\theta))},\qquad
  \tilde s(\theta):=\frac{\theta q_0(\theta)(p_{+}-q_1)(s_1-1)(1+\eps)}{q_1\eps(1-\theta)(p_{+}-q_0(\theta))}
\end{equation*}
and
\begin{equation*}
  \tilde t(\theta):=\frac{q_0(\theta)(p_{+}-q)(s-1)(1+\delta)}{q(1-\theta)(s_0(\theta)-1)(p_{+}-q_0(\theta))},
\end{equation*}
\begin{equation*}
  \tilde u(\theta):=\frac{\theta q_0(\theta)(p_{+}-q_1)(1+\delta)}{q_1\delta(1-\theta)(s_0(\theta)-1)(p_{+}-q_0(\theta))}.
\end{equation*}

Now, we choose $\eps=\frac{\theta q(p_{+}-q_1)(s_1-1)}{q_1(p_{+}-q)}$ and $\delta=\frac{\theta q(p_{+}-q_1)}{q_1(p_{+}-q)(s-1)}$ in such a way that
\begin{equation*}
  \tilde r(\theta)=\tilde s(\theta)=\frac{q_0(\theta)(q_1(p_{+}-q)+\theta q(p_{+}-q_1)(s_1-1))}{qq_1(1-\theta)(p_{+}-q_0(\theta))},
\end{equation*}
and
\begin{equation*}
  \tilde t(\theta)=\tilde u(\theta)=\frac{q_0(\theta)(q_1(p_{+}-q)(s-1)+\theta q(p_{+}-q_1))}{qq_1(1-\theta)(s_0(\theta)-1)(p_{+}-q_0(\theta))}.
\end{equation*}

The strategy to proceed is the same as in the proof of Lemma \ref{lem:main}. In particular, we use the reverse H\"older inequality (\ref{eq:RHI}) for $A_v(\R^d)$ weights.

Recalling that $q_0(0)=q$, we see that $\tilde r(0)=\tilde t(0)=1$. By continuity, given any $\eta>0$, we find that
\begin{equation*}
\max(\tilde r(\theta),\tilde t(\theta))\leq 1+\eta\,\,\,\text{for all small enough}\,\,\,\theta>0.
\end{equation*}

By property \eqref{weights prop. 1} of Proposition \ref{prop:weights} each of the four functions $v\in A_s(\R^d)$, $v^{-\frac{1}{s-1}}\in A_{s'}(\R^d)$, $v_1\in A_{s_1}(\R^d)$ and $v_1^{-\frac{1}{s_1-1}}\in A_{s'_1}(\R^d)$ satisfies the reverse H\"older inequality (\ref{eq:RHI}) for all $t\leq 1+\eta$ and for some $\eta>0$. Thus, for all small enough $\theta>0$, we have
\begin{equation*}
\begin{split}
  \eqref{eq:beforeRHI2}
  &\lesssim  \ave{v}_Q^\frac{q_0(p_{+}-q)}{q(1-\theta)(p_{+}-q_0)} \ave{v_1^{-\frac{1}{s_1-1}}}_Q^\frac{\theta q_0(p_{+}-q_1)(s_1-1)}{q_1(1-\theta)(p_{+}-q_0)}  \\
  &\qquad\times\ave{v^{-\frac{1}{s-1}}}_Q^{\frac{q_0(p_{+}-q)(s-1)}{q(1-\theta)(p_{+}-q_0)}} \ave{v_1}_Q^\frac{\theta q_0(p_+-q_1)}{q_1(1-\theta)(p_{+}-q_0)}  \\
  &=(\ave{v}_Q \ave{v^{-\frac{1}{s-1}}}_Q^{s-1})^\frac{q_0(p_{+}-q)}{q(1-\theta)(p_{+}-q_0)}  \\
  &\qquad\times(\ave{v_1}_Q \ave{v_1^{-\frac{1}{s_1-1}}}_Q^{s_1-1})^\frac{\theta q_0(p_{+}-q_1)}{q_1(1-\theta)(p_{+}-q_0)}  \\
  &\leq[v]_{A_s}^\frac{q_1(p_{+}-q)}{p_{+}(q_1-\theta q)-qq_1(1-\theta)}[v_1]_{A_{s_1}}^\frac{\theta q(p_{+}-q_1)}{p_{+}(q_1-\theta q)-qq_1(1-\theta)}.
\end{split}
\end{equation*}
In combination with the lines preceding \eqref{eq:beforeRHI2}, we have shown that
\begin{equation*}
  [v_0]_{A_{s_0}}\lesssim [v]_{A_s}^\frac{q_1(p_{+}-q)}{p_{+}(q_1-\theta q)-qq_1(1-\theta)}[v_1]_{A_{s_1}}^\frac{\theta q(p_{+}-q_1)}{p_{+}(q_1-\theta q)-qq_1(1-\theta)}<\infty,
\end{equation*}
provided that $\theta>0$ is small enough. This concludes the proof in the case $q_1\in(p_{-},p_{+})$. 

The case of $q_1=p_{-}$ (thus $v_1:=w_1^{(p_{+}/q_1)'}\in A_{1}(\R^d)$) follows by similar computations but with the main difference that in the first average of \eqref{eq:beforerhi}, we do not use H\"older's inequality. In particular, this average is bounded from above by
\begin{equation*}
\begin{split}
  &\leq\ave{w^{\frac{q_0(p_{+}/q_0)'}{q(1-\theta)}}}_Q\Norm{w_1^{-\frac{q_0\cdot\theta(p_{+}/q_0)'}{q_1(1-\theta)}}}{L^\infty(Q)}  \\
  &=\ave{(w^{(p_{+}/q)'})^{\tilde r(\theta)}}_Q\Norm{(w_1^{-(p_{+}/q_1)'})^{\tilde s(\theta)}}{L^\infty(Q)} \\
  &=\ave{v^{\tilde r(\theta)}}_Q\Norm{v_1^{-1}}{L^\infty(Q)}^{\tilde s(\theta)}  \\
  &\lesssim \ave{v}_Q^{\tilde r(\theta)}\Norm{v_1^{-1}}{L^\infty(Q)}^{\tilde s(\theta)},
\end{split}
\end{equation*}
where 
\begin{equation*}
  \tilde r(\theta):=\frac{q_0(\theta)(p_{+}-q)}{q(1-\theta)(p_{+}-q_0(\theta))},\qquad
  \tilde s(\theta):=\frac{\theta q_0(\theta)(p_{+}-q_1)}{q_1(1-\theta)(p_{+}-q_0(\theta))}
\end{equation*}
and in the last step we used the reverse H\"older inequality (\ref{eq:RHI}) for the weight $v\in A_s(\R^d)$. The rest of the proof concerning the second average of \eqref{eq:beforerhi} remains the same as in the case $q_1\in(p_{-},p_{+})$.

Now, let us assume that $q_1=p_{+}$ so that $w_1\in A_{p_{+}/p_{-}}(\R^d)\cap RH_{\infty}(\R^d)$. The first condition of the lemma is proved in the same way as in the previous cases. We only need to check that $w_0^{(p_{+}/q_0)'}\in A_{s_0}(\R^d)$, so we consider a cube $Q$ and write

\begin{equation*}
\begin{split}
  \ave{w_0^{(p_{+}/q_0)'}}_Q\ave{w_0^{(p_{+}/q_0)'(-\frac{1}{s_0-1})}}_Q^{s_0-1}&=\ave{w^{\frac{q_0(p_{+}/q_0)'}{q(1-\theta)}}w_1^{-\frac{q_0\cdot\theta(p_{+}/q_0)'}{p_{+}(1-\theta)}}}_Q  \\
  &\qquad\times\ave{w^{-\frac{q_0(p_{+}/q_0)'}{q(1-\theta)(s_0-1)}}w_1^{\frac{q_0\cdot\theta(p_{+}/q_0)'}{p_{+}(1-\theta)(s_0-1)}}}_Q^{s_0-1}.
\end{split}
\end{equation*}

In the first average, we use H\"older's inequality with exponents $1+\eps^{\pm 1}$ and we bound from above the second average as follows
\begin{equation}\label{eq1:beforeRHI2}
\begin{split}
  &\leq \ave{w^{\frac{q_0(p_{+}/q_0)'(1+\eps)}{q(1-\theta)}}}_Q^\frac{1}{1+\eps} \ave{w_1^{-\frac{q_0\theta(p_{+}/q_0)'(1+\eps)}{p_{+}\eps(1-\theta)}}}_Q^\frac{\eps}{1+\eps}   \\ 
  &\qquad\times\ave{w^{-\frac{q_0(p_{+}/q_0)'}{q(1-\theta)(s_0-1)}}}_Q^{s_0-1} \Norm{w_1^{\frac{q_0\theta(p_{+}/q_0)'}{p_{+}(1-\theta)(s_0-1)}}}{L^\infty(Q)}^{s_0-1}  \\
  &=\ave{(w^{(p_{+}/q)'})^{\tilde r(\theta)}}_Q^\frac{1}{1+\eps} \ave{(w_1^{-\frac{1}{\frac{p_{+}}{p_{-}}-1}})^{\tilde s(\theta)}}_Q^\frac{\eps}{1+\eps}  \\
  &\qquad\times\ave{(w^{(p_{+}/q)'(-\frac{1}{s-1})})^{\tilde t(\theta)}}_Q^{s_0-1} \Norm{w_1}{L^\infty(Q)}^\frac{q_0\theta(p_{+}/q_0)'}{p_{+}(1-\theta)},
\end{split}
\end{equation}
where
\begin{equation*}
  \tilde r(\theta):=\frac{q_0(\theta)(p_{+}-q)(1+\eps)}{q(1-\theta)(p_{+}-q_0(\theta))},\qquad
  \tilde s(\theta):=\frac{\theta q_0(\theta)(p_{+}-p_{-})(1+\eps)}{p_{-}\eps(1-\theta)(p_{+}-q_0(\theta))}
\end{equation*}
and
\begin{equation*}
  \tilde t(\theta):=\frac{q_0(\theta)(p_{+}-q)(s-1)}{q(1-\theta)(s_0(\theta)-1)(p_{+}-q_0(\theta))}.
\end{equation*}

Now, we choose $\eps=\frac{\theta q(p_{+}-p_{-})}{p_{-}(p_{+}-q)}$ in such a way that
\begin{equation*}
  \tilde r(\theta)=\tilde s(\theta)=\frac{q_0(\theta)(p_{-}(p_{+}-q)+\theta q(p_{+}-p_{-}))}{qp_{-}(1-\theta)(p_{+}-q_0(\theta))}.
\end{equation*}

The strategy to proceed is the same as in the the previous cases. In particular, we use the reverse H\"older inequality (\ref{eq:RHI}) for $A_v(\R^d)$ weights.

Recalling that $q_0(0)=q$, we see that $\tilde r(0)=\tilde t(0)=1$. By continuity, given any $\eta>0$, we find that
\begin{equation*}
\max(\tilde r(\theta),\tilde t(\theta))\leq 1+\eta\,\,\,\text{for all small enough}\,\,\,\theta>0.
\end{equation*}

By property \eqref{weights prop. 1} of Proposition \ref{prop:weights} each of the three functions $w^{(p_{+}/q)'}\in A_{s}(\R^d)$, $w^{(p_{+}/q)'(-\frac{1}{s-1})}\in A_{s'}(\R^d)$ and $w_1^{-\frac{1}{\frac{p_{+}}{p_{-}}-1}}\in A_{(p_{+}/p_{-})'}(\R^d)$ satisfies the reverse H\"older inequality (\ref{eq:RHI}) for all $t\leq 1+\eta$ and for some $\eta>0$. Thus, for all small enough $\theta>0$, we have
\begin{equation*}
\begin{split}
  \eqref{eq1:beforeRHI2}
  &\lesssim  \ave{w^{(p_{+}/q)'}}_Q^\frac{q_0(p_{+}-q)}{q(1-\theta)(p_{+}-q_0)} \ave{w_1^{-\frac{1}{\frac{p_{+}}{p_{-}}-1}}}_Q^\frac{\theta q_0(p_{+}-p_{-})}{p_{-}(1-\theta)(p_{+}-q_0)}  \\
  &\qquad\times\ave{w^{(p_{+}/q)'(-\frac{1}{s-1})}}_Q^{\frac{q_0(p_{+}-q)(s-1)}{q(1-\theta)(p_{+}-q_0)}} \Norm{w_1}{L^\infty(Q)}^\frac{\theta q_0}{(1-\theta)(p_{+}-q_0)}  \\
  &=(\ave{w^{(p_{+}/q)'}}_Q \ave{w^{(p_{+}/q)'(-\frac{1}{s-1})}}_Q^{s-1})^\frac{q_0(p_{+}-q)}{q(1-\theta)(p_{+}-q_0)}  \\
  &\qquad\times(\Norm{w_1}{L^\infty(Q)} \ave{w_1^{-\frac{1}{\frac{p_{+}}{p_{-}}-1}}}_Q^{\frac{p_{+}}{p_{-}}-1})^\frac{\theta q_0}{(1-\theta)(p_{+}-q_0)}  \\
  &\leq[w^{(p_{+}/q)'}]_{A_s}([w_1]_{RH_{\infty}}[w_1]_{\frac{p_{+}}{p_{-}}})^\frac{\theta q}{p_{+}-q}.
\end{split}
\end{equation*}
In combination with the lines preceding \eqref{eq1:beforeRHI2}, we have shown that
\begin{equation*}
  [w_0^{(p_{+}/q_0)'}]_{A_{s_0}}\lesssim [w^{(p_{+}/q)'}]_{A_s}([w_1]_{RH_{\infty}}[w_1]_{\frac{p_{+}}{p_{-}}})^\frac{\theta q}{p_{+}-q}<\infty,
\end{equation*}
provided that $\theta>0$ is small enough. This concludes the proof in the case $q_1=p_{+}$. 
\end{proof}

\begin{remark}
Lemma \ref{lem:main2} remains true in the case $p_{+}=\infty$, provided that $q_1<\infty$. In this case the reverse H\"older conditions on $w,w_0,w_1$ are vacuous and the proof is the same as in Lemma \ref{lem:main}.
\end{remark}

We now have the last missing ingredients of the proofs of Theorems \ref{thm:RdFcompact}, \ref{thm:Off-dig.extrp.compact} and \ref{thm:limited range extrp.compact}:

\begin{proof}[Proof of Proposition \ref{prop:LpvInterm}]
With some $\lambda\in[1,\infty)$, we are given $q,q_1\in(\lambda,\infty)$ and weights $v\in A_{q/\lambda}(\R^d)$, $v_1\in A_{q_1/\lambda}(\R^d)$. By Lemma \ref{lem:main}, there is some $q_0\in(\lambda,\infty)$, a weight $v_0\in A_{q_0/\lambda}(\R^d)$, and $\theta\in(0,1)$ such that
\begin{equation*}
  \frac{1}{q}=\frac{1-\theta}{q_0}+\frac{\theta}{q_1},\qquad
  v^{\frac{1}{q}}=v_0^{\frac{1-\theta}{q_0}}v_1^{\frac{\theta}{q_1}}.
\end{equation*}
By Theorem \ref{thm:SW}, we then have $L^q(v)=[L^{q_0}(v_0),L^{q_1}(v_1)]_\theta$, as we claimed.
\end{proof}

\begin{proof}[Proof of Proposition \ref{prop:offdiagonal}]
We are given $1<p\leq q<\infty$, $1<p_1\leq q_1<\infty$, and weights $v\in A_{p,q}(\R^d)$, $v_1\in A_{p_1,q_1}(\R^d)$. By Lemma \ref{lem:main1}, there is some $1<p_0\leq q_0<\infty$, a weight $v_0\in A_{p_0,q_0}(\R^d)$, and $\theta\in(0,1)$ such that
\begin{equation}\label{eq:prop}
  \frac{1}{p}=\frac{1-\theta}{p_0}+\frac{\theta}{p_1},\qquad\frac{1}{q}=\frac{1-\theta}{q_0}+\frac{\theta}{q_1},\qquad
  v=v_0^{1-\theta}v_1^{\theta}.
\end{equation}
By Theorem \ref{thm:SW}, we then have 
\begin{equation*}
[L^{p_0}({v_0}^{p_0}),L^{p_1}({v_1}^{p_1})]_\theta=L^p(v^p)\,\,\,\text{and}\,\,\,[L^{q_0}({v_0}^{q_0}),L^{q_1}({v_1}^{q_1})]_\theta=L^q(v^q).
\end{equation*}
Moreover, by \eqref{eq:prop} the claim of the proposition follows.
\end{proof}

\begin{proof}[Proof of Proposition \ref{prop:limitedrange}]
We are given $1\leq p_{-}<p_{+}<\infty$, $q_1\in[p_{-},p_{+}]$, $q\in(p_{-},p_{+})$ and weights $v\in A_{q/p_{-}}(\R^d)\cap RH_{(p_{+}/q)'}(\R^d)$, $v_1\in A_{q_1/p_{-}}(\R^d)\cap RH_{(p_{+}/q_1)'}(\R^d)$. By Lemma \ref{lem:main2}, there is some $q_0\in(p_{-},p_{+})$, a weight $v_0\in A_{q_0/p_{-}}(\R^d)\cap RH_{(p_{+}/q_0)'}(\R^d)$, and $\theta\in(0,1)$ such that
\begin{equation*}
  \frac{1}{q}=\frac{1-\theta}{q_0}+\frac{\theta}{q_1},\qquad
  v^{\frac{1}{q}}=v_0^{\frac{1-\theta}{q_0}}v_1^{\frac{\theta}{q_1}}.
\end{equation*}
By Theorem \ref{thm:SW}, we then have $L^q(v)=[L^{q_0}(v_0),L^{q_1}(v_1)]_\theta$, as we claimed.
\end{proof}

\section{Calder\'on--Zygmund singular integrals}\label{sec:CZO}

In two applications below, we consider Calder\'on--Zygmund singular integral operators, or just Calder\'on--Zygmund operators for short, which are defined as follows: $T$ is a linear operator defined on a suitable class of test functions on $\R^d$, and it has the representation
\begin{equation*}
  Tf(x)=\int_{\R^d}K(x,y)f(y)\ud y,\qquad x\notin\supp f,
\end{equation*}
where the kernel $K$ satisfies the {\em standard estimates} 
\begin{equation*}
  \abs{K(x,y)}\lesssim\frac{1}{\abs{x-y}^d}
\end{equation*}
and, for some $\delta\in(0,1]$,
\begin{equation*}
  \abs{K(x,y)-K(x',y)}+\abs{K(y,x)-K(y,x')}\lesssim \frac{\abs{x-x'}^\delta}{\abs{x-y}^{d+\delta}},
\end{equation*}
for all $x,x',y\in\R^d$ such that $\abs{x-y}>\frac12\abs{x-x'}$. 

In our applications of Theorem \ref{thm:RdFcompact} to these operators, we never need to refer to the above definition; rather, we can apply several previous results for these operators as a black box. For our first application, we only need the following classical result of Coifman--Fefferman \cite{CF}. (See \cite{Lerner:simple} for a modern approach that also gives the sharp dependence on $[w]_{A_p}$ from \cite{Hytonen:A2}.)

\begin{theorem}[\cite{CF}]\label{thm:CF}
Let $T$ be a Calder\'on--Zygmund operator that extends to a bounded operator on $L^2(\R^d)$. Then $T$ extends to a bounded operator on $L^p(w)$ for all $p\in(1,\infty)$ and all $w\in A_p(\R^d)$.
\end{theorem}

We can now give a quick proof and a minor extension of a very recent result of Stockdale--Villarroya--Wick \cite{SVD}, which deals with $p_1=2$ and $w_1\equiv 1$:

\begin{corollary}[\cite{SVD}, Theorem 1.1]\label{cor:SVD}
Let $T$ be a Calder\'on--Zygmund operator that extends compactly to $L^{p_1}(w_1)$ for some $p_1\in(1,\infty)$ and some $w_1\in A_{p_1}(\R^d)$. Then $T$ extends to a compact operator on $L^p(w)$ for all $p\in(1,\infty)$ and all $w\in A_p(\R^d)$.
\end{corollary}

\begin{proof}
We verify the assumptions of Theorem \ref{thm:RdFcompact} for $\lambda=1$, and the exponent $p_1$ and weight $w_1$ appearing in the statement of the Corollary:
By Theorem \ref{thm:CF}, $T$ extends to a bounded operator on $L^{p_1}(\tilde w)$ for all $\tilde w\in A_{p_1}(\R^d)$.
By assumption, $T$ is compact on $L^{p_1}(w_1)$ for some $w_1\in A_{p_1}(\R^d)$.
Thus Theorem \ref{thm:RdFcompact} applies to give the compactness of $T$ on $L^p(w)$ for all $p\in(1,\infty)$ and all $w\in A_p(\R^d)$.
\end{proof}

The proof in \cite{SVD} was based on two quite recent ingredients: the technique of {\em sparse domination} of Calder\'on--Zygmund operators, which was essentially started by Lerner \cite{Lerner:simple} and thereafter extensively developed by many authors, together with a {\em characterisation of compactness} of Calder\'on--Zygmund operators due to Villarroya \cite{Villar:15}. We avoid all this.

\section{Commutators with functions of bounded mean oscillation}

Our several subsequent applications of Theorem \ref{thm:RdFcompact} deal with commutators of the form
\begin{equation*}
  [b,T]:f\mapsto bT(f)-T(bf),
\end{equation*}
where the pointwise multiplier $b$ belongs to the space
\begin{equation*}
  \BMO(\R^d):=\Big\{f:\R^d\to\C\ \Big|\ \Norm{f}{\BMO}:=\sup_Q\ave{\abs{f-\ave{f}_Q}}_Q<\infty\Big\}
\end{equation*}
 of functions of bounded mean oscillation, or its subspace
 \begin{equation*}
  \CMO(\R^d):=\overline{C_c^\infty(\R^d)}^{\BMO(\R^d)},
\end{equation*}
where the closure in the $\BMO$ norm. This subspace is also denoted by $\VMO(\R^d)$ in some papers; however, a distinction between the two notions is made in \cite{Bourdaud}.  There appears to be some confusion about these spaces in the literature, resulting from the fact that not all ``natural'' definitions lead to the same space. 

A powerful general result about commutators is due to \'Alvarez et al. \cite{ABKP}:

\begin{theorem}[\cite{ABKP}]\label{thm:commuBasic}
In the setting of Theorem \ref{thm:RdF}, suppose also that $b\in\BMO(\R^d)$. Then also $[b,T]$ extends to a bounded linear operator on $L^{p_1}(\tilde w)$ for all $\tilde w\in A_{p_1/\lambda}(\R^d)$, and its operator norm is dominated by another increasing function of $[\tilde w]_{A_{p_1/\lambda}}$. 
\end{theorem}

The statement in Theorem 2.13 of \cite{ABKP} is somewhat more general, but the above particular case is easily seen to be contained in it. In fact, the idea behind this general theorem is already implicit in the work of Coifman--Rochberg--Weiss \cite{CRW1976}, which deals with the case of Calder\'on--Zygmund operators.  See also \cite{CPP} for a sharp quantitative version, and \cite{BMMST} for an up-to-date treatment of this type of results.

A combination of Theorems \ref{thm:RdFcompact} and \ref{thm:commuBasic} immediately gives:

\begin{corollary}\label{cor:commuBasic}
Let $1\leq\lambda<p_1<\infty$, and let $T$ be a linear operator defined and bounded on $L^{p_1}(\tilde w)$ for all $\tilde w\in A_{p_1/\lambda}(\R^d)$, with the operator norm dominated by some increasing function of $[\tilde w]_{A_{p_1/\lambda}}$. Suppose moreover that the commutator $[b,T]$ is compact on $L^{p_1}(w_1)$ for some $b\in\BMO(\R^d)$ and some $w_1\in A_{p_1/\lambda}(\R^d)$. Then $[b,T]$ is compact for the same $b$, and for all $p\in(\lambda,\infty)$ and all $w\in A_{p/\lambda}(\R^d)$.
\end{corollary}

\begin{proof}
We verify the assumptions of Theorem \ref{thm:RdFcompact} for the numbers $\lambda,p_1$ and the weight $w_1$ appearing in the statement of the Corollary, and the operator $[b,T]$ in place of $T$. By Theorem \ref{thm:commuBasic}, $[b,T]$ is bounded on $L^{p_1}(\tilde w)$ for all $\tilde w\in A_{p_1/\lambda}(\R^d)$. By assumption, $[b,T]$ is compact on $L^{p_1}(w_1)$ for some $w_1\in A_{p_1/\lambda}(\R^d)$. Thus the assumptions, and hence the conclusion, of Theorem \ref{thm:RdFcompact} hold for the operator $[b,T]$ in place of $T$, and this is what we claimed.
\end{proof}

\section{Commutators of Calder\'on--Zygmund operators}\label{sec:comCZO}

The following result of Uchiyama \cite{Uchi}, based on a classical Fr\'echet--Kolmo\-gorov criterion for compactness in $L^p(\R^d)$, provides a concrete condition to verify the assumptions of Corollary \ref{cor:commuBasic}:

\begin{theorem}[\cite{Uchi}]\label{thm:Uchi}
Let $T$ be a Calder\'on--Zygmund operator that extends to a bounded operator on $L^2(\R^d)$. If $b\in\CMO(\R^d)$, then $[b,T]$ is compact on the unweighted $L^p(\R^d)$ for all $p\in(1,\infty)$.
\end{theorem}

Our next application of Theorem \ref{thm:RdFcompact} (via Corollary \ref{cor:commuBasic}) is a quick proof of the following result of  Clop--Cruz \cite{ClopCruz}, which they used to obtain weighted estimates for Beltrami equations. This result has also inspired a fair number of follow-up works dealing with the compactness of commutators in different settings, and we refer the reader to the several papers citing \cite{ClopCruz} for this.

\begin{corollary}[\cite{ClopCruz}, Theorem 2]\label{cor:ClopCruz}
Let $b\in\CMO(\R^d)$, and let $T$ be a Calder\'on\\--Zygmund operator that extends boundedly to $L^2(\R^d)$. Then the commutator $[b,T]$ is compact on $L^p(w)$ for all $p\in(1,\infty)$ and all $w\in A_p(\R^d)$.
\end{corollary}

\begin{proof}
Let us fix some $p_1\in(1,\infty)$ (any choice will do) for which we verify the assumptions of Corollary \ref{cor:commuBasic} with $\lambda=1$.
By Theorem \ref{thm:CF}, $T$ extends to a bounded operator on $L^{p_1}(\tilde w)$ for all $\tilde w\in A_{p_1}(\R^d)$.
By Theorem \ref{thm:Uchi}, $[b,T]$ is a compact operator on $L^{p_1}(\R^d)=L^{p_1}(w_1)$ with $w_1\equiv 1\in A_{p_1}(\R^d)$.
Thus Corollary \ref{cor:commuBasic} applies to give the compactness of $[b,T]$ on $L^p(w)$ for all $p\in(1,\infty)$ and all $w\in A_p(\R^d)$.
\end{proof}

The original proof in \cite{ClopCruz} relied on finding and verifying a weighted analogue of the classical (unweighted) Fr\'echet--Kolmogorov criterion, providing a sufficient condition for compactness in $L^p(w)$. This is avoided by the soft argument above.

\begin{remark}\label{rmk:pseudodiff}
Corollary \ref{cor:ClopCruz} is in particular valid for {\em pseudo-differential operators} with symbol of class $S^0_{1,0}$, namely (denoting by $\hat f$ the Fourier transform of $f$)
\begin{equation*}
  Tf(x)=\int_{\R^d}\sigma(x,\xi)\hat f(\xi)e^{i2\pi x\cdot\xi}\ud\xi,\quad
  \abs{\partial_x^\alpha\partial_\xi^\beta\sigma(x,\xi)}\lesssim(1+\abs{\xi})^{-\abs{\beta}}\ \forall\alpha,\beta\in\N^d,
\end{equation*}
as these are instances of Calder\'on--Zygmund operators by Th\'eor\`eme 19 in \cite{CM:audela}. For these $T$, Corollary \ref{cor:ClopCruz} is also obtained by Guo--Zhou \cite{GZ} as a corollary to Theorem 2.1 in \cite{GZ}, which allows a certain larger weight class $A_p(\varphi)$ introduced by Tang \cite{Tang}. Recovering Theorem 2.1 in \cite{GZ} by our approach requires its extension to $A_p(\varphi)$ in place of $A_p(\R^d)$. We refer to Section \ref{sec:Ap(varphi) weights} for detailed results concerning extrapolation with $A_p(\varphi)$ weights.
\end{remark}

\section{Commutators of rough singular integrals}\label{sec:RSIO}
Let us now consider
\begin{equation*}
  T_{\Omega}f(x):=\lim_{\eps\to 0}\int_{\abs{x-y}>\eps}\frac{\Omega(x-y)}{\abs{x-y}^d}f(y)\ud y,
\end{equation*}
where $\Omega$ is homogeneous of order zero, and integrable with vanishing mean on the unit sphere $S^{d-1}$. There are the following analogues of Theorems \ref{thm:CF} and \ref{thm:Uchi}:

\begin{theorem}[\cite{Duo:TAMS,Watson}]\label{thm:DW}
Let $r\in(1,\infty)$ and let $\Omega\in L^r(S^{d-1})$  be homogeneous of order zero with vanishing mean on $S^{d-1}$. Then $T_\Omega$ extends to a bounded operator on $L^p(w)$ for all $p\in(r',\infty)$ and all $w\in A_{p/r'}(\R^d)$.
\end{theorem}

Here and below, $r':=r/(r-1)$ denotes the conjugate exponent.

\begin{theorem}[\cite{ChenHu,GuoHu:15}]\label{thm:Chinese}
Let $\Omega$ be homogeneous of order zero, and integrable with vanishing mean on $S^{d-1}$. Let $b\in\CMO(\R^d)$. Then  the commutator $[b,T_\Omega]$ acts as a compact operator
\begin{enumerate}
  \item\label{it:ChenHu} on the unweighted $L^p(\R^d)$ for all $p\in(\theta',\theta)$, provided that $\theta>2$ and
\begin{equation*}
  \sup_{\zeta\in S^{d-1}}\int_{S^{d-1}}\abs{\Omega(\eta)}\Big(\log\frac{1}{\abs{\zeta\cdot\eta}}\Big)^\theta\ud\eta<\infty
  \qquad\text{\cite{ChenHu}};
\end{equation*}
  \item\label{it:GuoHu}  on the power-weighted $L^p(w_\gamma)$ for all $p\in(1,\infty)$ and $\gamma\in(-1,p-1)$, where $w_\gamma(x)=\abs{x}^\gamma$, provided that $\Omega\in L(\log L)^2(S^{d-1})$ \cite{GuoHu:15}.
\end{enumerate}
\end{theorem}

A combination of Corollary \ref{cor:commuBasic} and Theorem \ref{thm:DW}, together with either version \eqref{it:ChenHu} or \eqref{it:GuoHu} of Theorem \ref{thm:Chinese}, gives a recent result of Guo--Hu \cite{GuoHu:16}:

\begin{corollary}[\cite{GuoHu:16}]\label{cor:Chinese}
Let $r\in(1,\infty)$ and let $\Omega\in L^r(S^{d-1})$ be homogeneous of order zero with vanishing mean on $S^{d-1}$. Let $b\in\CMO(\R^d)$. Then the commutator $[b,T_\Omega]$ is compact on $L^p(w)$ for all $p\in(r',\infty)$ and all $w\in A_{p/r'}(\R^d)$. 
\end{corollary}

\begin{remark}
Both Theorem \ref{thm:DW} and Corollary \ref{cor:Chinese} also hold for $r=\infty$: Of course $L^\infty(S^{d-1})\subset L^r(S^{d-1})$ for all $r\in(1,\infty)$, while if $w\in A_p(\R^d)$, then also $w\in A_{p/r'}(\R^d)$ for large $r\in(1,\infty)$ by the openness of the $A_p(\R^d)$ condition from \cite{CF}. See \cite{HRT} for a quantitative version of Theorem \ref{thm:DW} with $r=\infty$.
\end{remark}

\begin{proof}[Proof of Corollary \ref{cor:Chinese}]
We verify the assumptions of Corollary \ref{cor:commuBasic} with $\lambda=r'$, a suitable $p_1\in(r',\infty)$ to be specified shortly, and $w_1\equiv 1\in A_{p_1/r'}(\R^d)$. It is clear that $\Omega\in L^r(S^{d-1})$ satisfies both conditions \eqref{it:ChenHu} (with any $\theta\in(2,\infty)$) and \eqref{it:GuoHu} of Theorem \ref{thm:Chinese}. If we wish to apply \eqref{it:ChenHu}, we choose $\theta$ sufficiently large so that $(\theta',\theta)\cap(r',\infty)\neq\varnothing$, and then we pick $p_1$ from this intersection. If we wish to apply \eqref{it:GuoHu} instead, then we are free to pick any $p_1\in(1,\infty)$. In either case, the relevant version of Theorem \ref{thm:Chinese} guarantees that $[b,T_\Omega]$ is compact on $L^{p_1}(w_1)$ for the particular exponent $p_1\in(r',\infty)$ and weight $w_1\equiv 1\in A_{p_1/r'}(\R^d)$; in case \eqref{it:GuoHu}, we use $\gamma=0$. On the other hand, a direct application of Theorem \ref{thm:DW} shows that $T_\Omega$ is bounded on $L^{p_1}(\tilde w)$ for all $\tilde w\in A_{p_1/r'}(\R^d)$. Thus Corollary \ref{cor:commuBasic} applies to give the compactness of $[b,T_\Omega]$ on $L^p(w)$ for all $p\in(r',\infty)$ and all $w\in A_{p/r'}(\R^d)$.
\end{proof}

\section{Commutators of fractional integral operators}\label{comm. fr. int. op.}

In this section we will apply Theorem \ref{thm:Off-dig.extrp.compact} to the commutator $[b,I_\alpha]$, where for $0<\alpha<d$, the fractional integral operator or Riesz potential $I_\alpha$ is defined by
\begin{equation*}
  I_\alpha f(x)=\int_{\R^d}\frac{f(y)}{|x-y|^{d-\alpha}}dy.
\end{equation*}
The weighted norm inequalities for $I_\alpha$ were obtained by Muckenhoupt--Wheeden \cite{MW:74} and the sharp behavior in terms of the weight constants by Lacey--Moen--P\'erez--Torres \cite{Lacey2010}. The commutators of fractional integral operators and $\BMO$ functions were first studied by Chanillo \cite{SC1982}. In \cite{ST:91}, Segovia--Torrea obtained the following weighted commutator result (see Cruz-Uribe and Moen \cite{CM2012} for a sharp quantitative version):

\begin{theorem}[\cite{ST:91}]\label{thm:linear-fractional}
Fix $0<\alpha<d$, $1<p<d/\alpha$, and  $1/p-1/q=\alpha/d$. Suppose also that $b\in\BMO(\R^d)$. Then $[b,I_\alpha]:L^p(w^p)\to L^q(w^q)$ is a bounded linear operator for all $w\in A_{p,q}(\R^d)$.
\end{theorem}

For the application of our Theorem \ref{thm:Off-dig.extrp.compact} we need the following result of Wang \cite{W:58} about the compactness of the commutator $[b,I_\alpha]$:

\begin{theorem}[\cite{W:58}]\label{thm:compact-fractional}
If $b\in\CMO(\R^d)$, then $[b,I_\alpha]:L^p(\R^d)\to L^q(\R^d)$ is a compact operator, where $0<\alpha<d$, $1<p<d/\alpha$, and  $1/p-1/q=\alpha/d$.
\end{theorem}

A combination of Theorems \ref{thm:Off-dig.extrp.compact}, \ref{thm:linear-fractional} and \ref{thm:compact-fractional} immediately gives a quick proof of the following recent result of Wu--Yang \cite{WY:2018}:

\begin{corollary}[\cite{WY:2018}, Theorem 1.3]
Let $\alpha\in(0,d),p,q\in(1,\infty)$ with $\frac{1}{p}=\frac{1}{q}+\frac{\alpha}{d},w\in A_{p,q}(\R^d)$ and $b\in\CMO(\R^d)$. Then the commutator $[b,I_\alpha]$ is compact from $L^p(w^p)$ to $L^q(w^q)$.
\end{corollary}

\begin{proof}
Let us fix some $p_1,q_1\in(1,\infty)$ for which we verify the assumptions of Theorem \ref{thm:Off-dig.extrp.compact} for $[b,I_\alpha]$ in place of $T$: By Theorem \ref{thm:linear-fractional}, $[b,I_\alpha]$ is a bounded operator from $L^{p_1}(\tilde w^{p_1})$ to $L^{q_1}(\tilde w^{q_1})$ for all $1<p_1\leq q_1<\infty$ such that $\frac{1}{p}-\frac{1}{q}=\frac{1}{p_1}-\frac{1}{q_1}=\frac{\alpha}{d}$ and all $\tilde w\in A_{p_1,q_1}(\R^d)$. By Theorem \ref{thm:compact-fractional}, $[b,I_\alpha]$ is a compact operator from $L^{p_1}(\R^d)=L^{p_1}(w_1^{p_1})$ to $L^{q_1}(\R^d)=L^{q_1}(w_1^{q_1})$ with $w_1\equiv 1\in A_{p_1,q_1}(\R^d)$. Thus the assumptions, and hence the conclusion, of Theorem \ref{thm:Off-dig.extrp.compact} hold for the operator $[b,I_\alpha]$ in place of $T$, and this is what we wanted.
\end{proof}

As in the case of the commutators of Calder\'on--Zygmund operators in Section \ref{sec:comCZO} the original proof in \cite{WY:2018} relied on verifying the weighted Fr\'echet--Kolmog\-orov criterion \cite{ClopCruz}. This is avoided by the aforementioned argument.

Consider now, for $\alpha\in(0,d)$, the so-called {\em $\rho$-type fractional integral operator} defined by
\begin{equation*}
  T_{K_{\alpha}}f(x)=\int_{\R^d}K_{\alpha}(x,y)f(y)dy,\qquad x\notin\supp f,
\end{equation*}
with kernel $K_{\alpha}$ satisfying the size condition
\begin{equation*}
  |K_\alpha(x,y)|\lesssim\frac{1}{|x-y|^{d-\alpha}},
\end{equation*}
and the smoothness condition
\begin{equation*}
  |K_{\alpha}(x,y)-K_{\alpha}(z,y)|+|K_{\alpha}(y,x)-K_{\alpha}(y,z)|\leq\rho\bigg(\frac{|x-z|}{|x-y|}\bigg)\frac{1}{|x-y|^{d-\alpha}},
\end{equation*}
for all $x,z,y\in\R^d$ such that $|x-y|>2|x-z|$, where $\rho:[0,1]\to[0,\infty)$ is a modulus of continuity, that is, $\rho$ is a continuous, increasing, subadditive function with $\rho(0)=0$ and satisfies the following Dini condition:
\begin{equation*}
  \int_{0}^{1}\rho(t)\frac{dt}{t}<\infty.
\end{equation*}

By observing that $\abs{T_{K_\alpha}(f)}\lesssim I_{\alpha}(|f|)$ and applying the result of Muckenho\-upt and Wheeden \cite{MW:74} to the operator $I_{\alpha}(|f|)$ we have that $T_{K_\alpha}$ is bounded from $L^p(w^p)$ to $L^q(w^q)$ for all $1<p\leq q<\infty$ such that $\frac{1}{p}-\frac{1}{q}=\frac{\alpha}{d}$ and all weights $w\in A_{p,q}(\R^d)$. We extend this result to the commutator $[b,T_{K_\alpha}]$ by recalling the following result of B\'enyi--Martell--Moen--Stachura--Torres \cite{BMMST} (this is a generalized version of the classical theorem of Coifman--Rochberg--Weiss \cite{CRW1976}):

\begin{theorem}[\cite{BMMST}, Theorem 3.22]\label{thm:commutators lim. range}
Let $T$ be a linear operator. Fix $1\leq p, q<\infty$. Suppose also that $T:L^p(w^p)\to L^q(w^q)$ is bounded for all $w\in A_{p,q}(\R^d)$ and $b\in\BMO(\R^d)$. Then $[b,T]$ is a bounded operator from $L^p(w^p)$ to $L^q(w^q)$.
\end{theorem}

By applying Theorem \ref{thm:commutators lim. range} to the operator $T_{K_\alpha}$ in place of $T$, the following weighted boundedness result for the commutator $[b,T_{K_\alpha}]$ is automatically valid:

\begin{corollary}\label{coro:rho type fractional}
Fix $0<\alpha<d$, $1<p<d/\alpha$ and $1/p-1/q=\alpha/d$. Suppose also that $b\in\BMO(\R^d)$. Then $[b,T_{K_\alpha}]:L^p(w^p)\to L^q(w^q)$ is a bounded linear operator for all $w\in A_{p,q}(\R^d)$.
\end{corollary}

The compactness result about the commutator $[b,T_{K_\alpha}]$ is due to Guo--Wu--Yang \cite{GWY}:

\begin{theorem}[\cite{GWY}, Theorem 1.5]\label{thm:cmpt. of rho type fractional}
Let $w\in A_{p,q}(\R^d)$, $1<p,q<\infty$, $0<\alpha<d$, $1/q=1/p-\alpha/d$. If $b\in\CMO(\R^d)$, then $[b,T_{K_\alpha}]$ is a compact operator from $L^p(w^p)$ to $L^q(w^q)$.
\end{theorem}

The original proof of Theorem \ref{thm:cmpt. of rho type fractional} again follows by applying the weighted Fr\'echet--Kolmogorov criterion obtained in \cite{ClopCruz} and restated in Lemma 5.4 of \cite{GWY}. However, by only applying and verifying the unweighted Fr\'echet--Kolmogorov criterion the proof of Theorem \ref{thm:cmpt. of rho type fractional} can be simplified as follows:

\begin{proof}
Let us fix some $p_1,q_1\in(1,\infty)$ for which we verify the assumptions of Theorem \ref{thm:Off-dig.extrp.compact} for $[b,T_{K_\alpha}]$ in place of $T$: By Corollary \ref{coro:rho type fractional}, $[b,T_{K_\alpha}]$ is a bounded operator from $L^{p_1}(\tilde w^{p_1})$ to $L^{q_1}(\tilde w^{q_1})$ for all $1<p_1\leq q_1<\infty$ such that $\frac{1}{p}-\frac{1}{q}=\frac{1}{p_1}-\frac{1}{q_1}=\frac{\alpha}{d}$ and all $\tilde w\in A_{p_1,q_1}(\R^d)$.
By the unweighted version of Theorem 1.5 in \cite{GWY} (which in turn depends on the classical, unweighted version of the Fr\'echet--Kolmogorov criterion), $[b,T_{K_\alpha}]$ is a compact operator from $L^{p_1}(\R^d)=L^{p_1}({w_1}^{p_1})$ to $L^{q_1}(\R^d)=L^{q_1}({w_1}^{q_1})$ with $w_1\equiv 1\in A_{p_1,q_1}(\R^d)$. Thus the assumptions, and hence the conclusion of Theorem \ref{thm:Off-dig.extrp.compact} hold for the operator $[b,T_{K_\alpha}]$ in place of $T$, and this is what we wanted.
\end{proof}

\section{Commutators of Bochner--Riesz multipliers}\label{comm. br. mult.}
In this section we will apply Theorem \ref{thm:limited range extrp.compact} to the commutators of Bochner--Riesz multipliers in dimensions $d\ge 2$. The latter is a Fourier multiplier $B^\kappa$ with the symbol $(1-|\xi|^2)_{+}^{\kappa}$, where $\kappa>0$ and $t_{+}=\max(t,0)$. That is, the Bochner--Riesz operator is defined, on the class $\mathcal S(\R^d)$ of Schwartz functions, by
\begin{equation*}
  \widehat {B^\kappa f} (\xi)=(1-|\xi|^2)_{+}^{\kappa}\widehat f(\xi),
\end{equation*}
where $\widehat f$ denotes the Fourier transform of $f$.

The following Bochner--Riesz conjecture is well-known:

\begin{conjecture}[Bochner--Riesz Conjecture]\label{j:BR} 
For $0<\kappa<\frac{d-1}{2}$, there holds that $B^\kappa: L ^{p}(\R ^{d})\mapsto L ^{p}(\R^{d})$ if 
\begin{equation*}
  p\in\bigg(\frac{2d}{d+1+2\kappa},\frac{2d}{d-1-2\kappa}\bigg). 
\end{equation*}
\end{conjecture}

This conjecture holds in two dimensions, as was proved by Carleson--Sj\"olin \cite{CS} (see also C\'ordoba \cite{Cordoba1979}). In the case $d\ge 3$, the best results are currently due to Bourgain--Guth \cite{BG}, but also see Lee \cite{Lee}.

In \cite{LMR2019}, an equivalent form of the Bochner--Riesz Conjecture \ref{j:BR} is stated as follows:
\begin{conjecture}\label{conj.}
Let $\mathbf 1_{[-1/4,1/4]}\leq\chi\leq\mathbf 1_{[-1/2,1/2]}$ be a Schwartz function and denote by $S_{\tau}$ the Fourier multiplier with symbol  $\chi((|\xi|-1)/\tau)$. If $\frac{2d}{d+1}<p<\frac{2d}{d-1}$, then 
\begin{equation}\label{eq:conjecture}
  \|S_{\tau}\|_{L^p(\R^d)\mapsto L^p(\R^d)}\leq C_{\epsilon}\tau^{-\epsilon},
\end{equation}
where $0<\tau<1$ and $C_{\epsilon}$ is a constant that depends on $0<\epsilon<1$.
\end{conjecture}

The connection between the Bochner--Riesz and the $S_{\tau}$ Fourier multipliers is well-known and it can be found in \cite{Carbery, Cordoba1977, Cordoba1979} and \cite[Chapter 8.5]{D2001}. We briefly recall it here. For each $0<\kappa<\frac{d-1}{2}$, we have
\begin{equation*}
  B^\kappa=T^0+\sum_{i=1}^{\infty }2^{-i\kappa}\operatorname{Dil}_{1-2^{-i}}S_{2^{-i}}, 
\end{equation*}
where $T^0$ is a Fourier multiplier, with the multiplier being a Schwartz function supported near the origin and the operator $\operatorname{Dil}_s f(x)=f(x/s)$ is a dilation operator. 
Moreover, each $S_{2^{-i}}$ is a Fourier multiplier with symbol $\chi_i (2^{i}\big||\xi|-1\big|)$, where the $\chi_i$ satisfy a uniform class of derivative estimates.

The partial knowledge of the range of exponents (which depends on the parameter $1<p_0<2$) such that the estimate \eqref{eq:conjecture} of Conjecture \ref{conj.} holds is used in the following theorem of Lacey--Mena--Reguera \cite{LMR2019}:

\begin{theorem}[\cite{LMR2019}, Theorem 6.1]\label{thm:Bochner-Riesz limited range estimate}
If $d=2$, $0<\kappa<\tilde\kappa<\frac{1}{2}$ and $p\in(\frac{4}{1+6\kappa},\frac{4}{1-2\kappa})$, then $B^{\tilde\kappa}$ is bounded on $L^p(w)$ for all $w\in A_{\frac{p(1+6\kappa)}{4}}(\R^2)\cap RH_{\big(\frac{4}{p(1-2\kappa)}\big)'}(\R^2)$.
\newline Moreover, if $d\ge 3$, $0<\kappa<\tilde\kappa<\frac{d-1}{2}$ and $1<p_0<2$ are such that the estimate \eqref{eq:conjecture} of Conjecture \ref{conj.} holds, and 
\begin{equation*}
  p\in\bigg(\frac{p_0(d-1)}{d-1+2\kappa(p_0-1)},\frac{p_0(d-1)}{d-1-2\kappa}\bigg),
\end{equation*}
then $B^{\tilde\kappa}$ is bounded on $L^p(w)$ for all
\begin{equation*}
  w\in A_{\frac{p(d-1+2\kappa(p_0-1))}{p_0(d-1)}}(\R^d)\cap RH_{\big(\frac{p_0(d-1)}{p(d-1-2\kappa)}\big)'}(\R^d).
\end{equation*}
\end{theorem}

Some earlier results in the same direction are contained in \cite{BBL}, \cite{CDL} and \cite{Christ}.

To streamline the presentation of our main result in this section about the compactness of commutators of Bochner--Riesz multipliers, we formulate the following Corollary of Theorem \ref{thm:Bochner-Riesz limited range estimate}:

\begin{corollary}\label{coro:Bochner-Riesz limited range estimate}
If $d=2$, $0<\tilde\kappa<\frac{1}{2}$ and $p\in(\frac{4}{1+6\tilde\kappa},\frac{4}{1-2\tilde\kappa})$, then $B^{\tilde\kappa}$ is bounded on $L^p(w)$ for all $w\in A_{\frac{p(1+6\tilde\kappa)}{4}}(\R^2)\cap RH_{\big(\frac{4}{p(1-2\tilde\kappa)}\big)'}(\R^2)$.
\newline Moreover, if $d\ge 3$, $0<\tilde\kappa<\frac{d-1}{2}$ and $1<p_0<2$ are such that the estimate \eqref{eq:conjecture} of Conjecture \ref{conj.} holds, and 
\begin{equation*}
  p\in\bigg(\frac{p_0(d-1)}{d-1+2\tilde\kappa(p_0-1)},\frac{p_0(d-1)}{d-1-2\tilde\kappa}\bigg),
\end{equation*}
then $B^{\tilde\kappa}$ is bounded on $L^p(w)$ for all
\begin{equation*}
  w\in A_{\frac{p(d-1+2\tilde\kappa(p_0-1))}{p_0(d-1)}}(\R^d)\cap RH_{\big(\frac{p_0(d-1)}{p(d-1-2\tilde\kappa)}\big)'}(\R^d).
\end{equation*} 
\end{corollary}

\begin{proof}
Let us fix $\tilde\kappa, p$ and the weight $w$ of our assumptions. For each selection of these fixed values we show that we can choose $\kappa$ sufficiently close to $\tilde\kappa$ (depending on $\tilde\kappa, p$ and the weight $w$) such that the assumptions of Theorem \ref{thm:Bochner-Riesz limited range estimate} are satisfied. By properties \eqref{weights prop. 2} and \eqref{weights prop. 3} of Proposition \ref{prop:weights} if
\begin{equation*}
  w\in A_{\frac{p(d-1+2\tilde\kappa(p_0-1))}{p_0(d-1)}}(\R^d)\cap RH_{\big(\frac{p_0(d-1)}{p(d-1-2\tilde\kappa)}\big)'}(\R^d)
\end{equation*}
then for $\kappa$ sufficiently close to $\tilde\kappa$ we also have that $\frac{p(d-1+2\kappa(p_0-1))}{p_0(d-1)}$ is sufficiently close to $\frac{p(d-1+2\tilde\kappa(p_0-1))}{p_0(d-1)}$ and $\big(\frac{p_0(d-1)}{p(d-1-2\kappa)}\big)'$ is sufficiently close to
\\$\big(\frac{p_0(d-1)}{p(d-1-2\tilde\kappa)}\big)'$ such that 
\begin{equation*}
  w\in A_{\frac{p(d-1+2\kappa(p_0-1))}{p_0(d-1)}}(\R^d)\cap RH_{\big(\frac{p_0(d-1)}{p(d-1-2\kappa)}\big)'}(\R^d).
\end{equation*}

By continuity, since 
\begin{equation*}
  p\in\bigg(\frac{p_0(d-1)}{d-1+2\tilde\kappa(p_0-1)},\frac{p_0(d-1)}{d-1-2\tilde\kappa}\bigg)
\end{equation*}
we also have that 
\begin{equation*}
  p\in\bigg(\frac{p_0(d-1)}{d-1+2\kappa(p_0-1)},\frac{p_0(d-1)}{d-1-2\kappa}\bigg),
\end{equation*}
provided that $\kappa$ is sufficiently close to $\tilde\kappa$.

Hence the assumptions of Theorem \ref{thm:Bochner-Riesz limited range estimate} are satisfied, and thus $B^{\tilde\kappa}$ is bounded on $L^p(w)$ for the arbitrary choice of the quantities $\tilde\kappa, p$ and $w$ in the statement of Corollary \ref{coro:Bochner-Riesz limited range estimate} that we considered. This concludes the proof.
\end{proof}

We extend this result to the commutator $[b,B^\kappa]$ by recalling the following corollary of Theorem \ref{thm:commutators lim. range} obtained in \cite{BMMST} (it follows by applying property \eqref{weights prop. 5} of Proposition \ref{prop:weights}):

\begin{corollary}[\cite{BMMST}, Corollary 5.3]\label{coro:restricted}
Let $1\leq p_{-}<p<p_{+}\leq\infty$, and let $T$ be a linear operator bounded on $L^{p}(w)$ for all $w\in A_{\frac{p}{p_{-}}}(\R^d)\cap RH_{\big(\frac{p_{+}}{p}\big)'}(\R^d)$. If $b\in\BMO(\R^d)$, then $[b,T]$ is bounded on $L^p(w)$ for all $w\in A_{\frac{p}{p_{-}}}(\R^d)\cap RH_{\big(\frac{p_{+}}{p}\big)'}(\R^d)$.
\end{corollary}

By applying Corollary \ref{coro:restricted} to the operator $B^\kappa$ in place of $T$, the following weighted boundedness for the commutator $[b,B^\kappa]$ holds:

\begin{corollary}\label{coro:commu. of Bochner-Riesz limited range estimate}
If $d=2$, $0<\kappa<\frac{1}{2}$, and $p\in(\frac{4}{1+6\kappa},\frac{4}{1-2\kappa})$, then for $b\in\BMO(\R^2)$, the commutator $[b,B^\kappa]$ is bounded on $L^p(w)$ for all $w\in A_{\frac{p(1+6\kappa)}{4}}(\R^2)\cap RH_{\big(\frac{4}{p(1-2\kappa)}\big)'}(\R^2)$.
\newline Moreover, if $d\ge 3$, $0<\kappa<\frac{d-1}{2}$ and $1<p_0<2$ are such that the estimate \eqref{eq:conjecture} of Conjecture \ref{conj.} holds and 
\begin{equation*}
  p\in\bigg(\frac{p_0(d-1)}{d-1+2\kappa(p_0-1)},\frac{p_0(d-1)}{d-1-2\kappa}\bigg),
\end{equation*}
then for $b\in\BMO(\R^d)$, the commutator $[b,B^\kappa]$ is bounded on $L^p(w)$ for all
\begin{equation*}
  w\in A_{\frac{p(d-1+2\kappa(p_0-1))}{p_0(d-1)}}(\R^d)\cap RH_{\big(\frac{p_0(d-1)}{p(d-1-2\kappa)}\big)'}(\R^d).
\end{equation*}
\end{corollary}

Moreover, an unweighted compactness result for the commutator $[b,B^\kappa]$ is due to Bu--Chen--Hu \cite{BCH}:

\begin{theorem}[\cite{BCH}, Theorems 1.1 and 1.2]\label{thm:unweighted compt. comm. of Bochner-Riesz limited range estimate}
If $d=2$, $0<\kappa<\frac{1}{2}$, and $p\in(\frac{4}{3+2\kappa},\frac{4}{1-2\kappa})$, then for $b\in\CMO(\R^2)$, the commutator $[b,B^\kappa]$ is compact on $L^p(\R^2)$.
\newline Let $d\ge3$, $\frac{d-1}{2d+2}<\kappa<\frac{d-1}{2}$, and $p\in(\frac{2d}{d+1+2\kappa},\frac{2d}{d-1-2\kappa})$. Then for $b\in\CMO(\R^d)$, the commutator $[b,B^\kappa]$ is compact on $L^p(\R^d)$.
\end{theorem}

Combining Theorem \ref{thm:limited range extrp.compact}, Corollary \ref{coro:commu. of Bochner-Riesz limited range estimate} and Theorem \ref{thm:unweighted compt. comm. of Bochner-Riesz limited range estimate} we can give a new weighted compactness result for the Bochner--Riesz commutator $[b,B^\kappa]$:

\begin{corollary}
If $d=2$, $0<\kappa<\frac{1}{2}$, and $p\in(\frac{4}{1+6\kappa},\frac{4}{1-2\kappa})$, then for $b\in\CMO(\R^2)$, the commutator $[b,B^\kappa]$ is compact on $L^p(w)$ for all $w\in A_{\frac{p(1+6\kappa)}{4}}(\R^2)\cap RH_{\big(\frac{4}{p(1-2\kappa)}\big)'}(\R^2)$.
\newline Moreover, if $d\ge 3$, $\frac{d-1}{2d+2}<\kappa<\frac{d-1}{2}$ and $1<p_0<2$ are such that the estimate \eqref{eq:conjecture} of Conjecture \ref{conj.} holds,
\begin{equation*}
  p\in\bigg(\frac{p_0(d-1)}{d-1+2\kappa(p_0-1)},\frac{p_0(d-1)}{d-1-2\kappa}\bigg),   
\end{equation*}
and $b\in\CMO(\R^d)$, then the commutator $[b,B^\kappa]$ is compact on $L^p(w)$ for all
\begin{equation*}
  w\in A_{\frac{p(d-1+2\kappa(p_0-1))}{p_0(d-1)}}(\R^d)\cap RH_{\big(\frac{p_0(d-1)}{p(d-1-2\kappa)}\big)'}(\R^d).
\end{equation*}
\end{corollary}

\begin{proof}
Let $d\ge 3$, $\frac{d-1}{2d+2}<\kappa<\frac{d-1}{2}$ and $p_0$ be as in the assumptions. We verify the assumptions of Theorem \ref{thm:limited range extrp.compact} for the fixed exponent
\begin{equation*}
  \frac{p_0(d-1)}{d-1+2\kappa(p_0-1)}<p_1<\frac{p_0(d-1)}{d-1-2\kappa}
\end{equation*}
and the operator $[b,B^\kappa]$ in place of $T$. By Corollary \ref{coro:commu. of Bochner-Riesz limited range estimate}, $[b,B^\kappa]$ is a bounded operator on $L^{p_1}(\tilde w)$ for all
\begin{equation*}
  \tilde w\in A_{\frac{p_1(d-1+2\kappa(p_0-1))}{p_0(d-1)}}(\R^d)\cap RH_{\big(\frac{p_0(d-1)}{p_1(d-1-2\kappa)}\big)'}(\R^d).
\end{equation*}
By Theorem \ref{thm:unweighted compt. comm. of Bochner-Riesz limited range estimate}, $[b,B^\kappa]$ is a compact operator on $L^{p_1}(\R^d)=L^{p_1}(w_1)$ with
\begin{equation*}
  w_1\equiv 1\in A_{\frac{p_1(d-1+2\kappa(p_0-1))}{p_0(d-1)}}(\R^d)\cap RH_{\big(\frac{p_0(d-1)}{p_1(d-1-2\kappa)}\big)'}(\R^d).
\end{equation*}
Thus Theorem \ref{thm:limited range extrp.compact} applies to give the compactness of $[b,B^\kappa]$ on $L^p(w)$ for all
\begin{equation*}
  p\in\bigg(\frac{p_0(d-1)}{d-1+2\kappa(p_0-1)},\frac{p_0(d-1)}{d-1-2\kappa}\bigg)
\end{equation*}
and all 
\begin{equation*}
  w\in A_{\frac{p(d-1+2\kappa(p_0-1))}{p_0(d-1)}}(\R^d)\cap RH_{\big(\frac{p_0(d-1)}{p(d-1-2\kappa)}\big)'}(\R^d).
\end{equation*}
The case $d=2$ follows in a similar way.
\end{proof}

\section{$A_p^\zeta(\varphi)$ weights and commutators of pseudo-differential operators}\label{sec:Ap(varphi) weights}

In this section, we develop and apply yet another variant for extrapolation of compactness for a special class of weights related to commutators of pseudo-differential operators with smooth symbols.

Following Wu--Wang \cite{WW:2018}, we consider the following:

\begin{definition}
A function $\varphi:[0,\infty)\to[1,\infty)$ is called {\em admissible} if it is non-decreasing and satisfies the following:
\begin{equation*}
  \varphi(\zeta t)\lesssim\zeta^{\omega}\varphi(t),
\end{equation*}
for all $\zeta\ge1$, $t\ge0$ and some $\omega>0$.
\end{definition}

\begin{definition}\label{def:Ap(varphi) weights}
Let $\varphi$ be an admissible function and let $p\in(1,\infty)$, $\zeta>0$. A weight $0<w\in L^1_{\loc}(\R^d)$ is called an $A_p^\zeta(\varphi)$ weight (or $w\in A_p^\zeta(\varphi)$) if
\begin{equation*}
  [w]_{A_p^\zeta(\varphi)}:=\sup_Q\frac{\ave{w}_Q\ave{w^{-\frac{1}{p-1}}}_Q^{p-1}}{\varphi(|Q|)^{\zeta p}}<\infty,
\end{equation*}
where the supremum is taken over all cubes $Q\subset\R^d$.
\end{definition}

\begin{remark}
In \cite{Tang}, Tang introduced the weight class $A_p(\varphi)$ which coincides with $A_p^1(\varphi)$. We remark that $A_p^\zeta(\varphi)=A_p(\varphi^\zeta)$. In general, it holds that $A_p(\R^d)\subset A_p^\zeta(\varphi)$ for all $1<p<\infty$. On the other hand, when $\varphi$ is a constant function, $A_p^\zeta(\varphi)=A_p(\R^d)$ for any $\zeta>0$. A main example of admissible function that we consider below is $\varphi(t)=1+t$.
\end{remark}

\subsection{Extrapolation with $A_p^\zeta(\varphi)$ weights}

In Theorem 2.1 of \cite{GZ}, Guo and Zhou proved the compactness of commutators of pseudo-differential operators with smooth symbols on weighted Lebesgue spaces where the weight functions belong to the weight class $A_p^\zeta(\varphi)$. Motivated by their work we show the following extrapolation of compactness:

\begin{theorem}\label{thm:Ap(varphi) compact}
Let $\varphi$ be an admissible function, let $1<p<\infty$, and let $T$ be a linear operator simultaneously defined and bounded on $L^{p}(w)$ for {\bf all} $1<p<\infty$, {\bf all} $w\in A_p^{\zeta}(\varphi)$ and {\bf all} $\zeta>0$. Suppose in addition that $T$ is compact on $L^{p_1}(w_1)$ for {\bf some} $1<p_1<\infty$, {\bf some} $w_1\in A_{p_1}^{\zeta_1}(\varphi)$ and {\bf some} $\zeta_1>0$. Then $T$ is compact on $L^p(w)$ for all $p\in (1,\infty)$, all $w\in A_p^{\zeta}(\varphi)$ and all $\zeta>0$.
\end{theorem}

We proceed by collecting the results from which the proof of Theorem \ref{thm:Ap(varphi) compact} follows. We will use Theorem \ref{thm:CwKa} in the special setting:

\begin{proposition}\label{prop:Ap(varphi)}
Let $\varphi$ be an admissible function and suppose that $q,q_1\in (1,\infty),\zeta,\zeta_1>0$, $v\in A_q^\zeta(\varphi)$ and $v\in A_{q_1}^{\zeta_1}(\varphi)$. Then
\begin{equation*}
  [L^{q_0}(v_0),L^{q_1}(v_1)]_{\gamma}=L^q(v)
\end{equation*}
for some $q_0\in(1,\infty), \zeta_0>0$, $v_0\in A_{q_0}^{\zeta_0}(\varphi)$, and $\gamma\in(0,1)$.
\end{proposition}

The only component of the proof of Theorem \ref{thm:Ap(varphi) compact} that requires actual computations is the verification of this proposition. For this purpose we will need Theorem \ref{thm:SW} which we connect with the $A_p^\zeta(\varphi)$ weights as follows:

\begin{lemma}\label{lem:main3}
Let $\varphi$ be an admissible function, and suppose that $p_1,p\in(1,\infty),\zeta,\zeta_1>0$, $w_1\in A_{p_1}^{\zeta_1}(\varphi)$ and $w\in A_p^{\zeta}(\varphi)$. Then there exist $p_0\in(1,\infty),\zeta_0>0$, $w_0\in A_{p_0}^{\zeta_0}(\varphi)$, and $\theta\in(0,1)$ such that the conclusion of Theorem \ref{thm:SW} holds, i.e.,
\begin{equation*}
  [L^{p_0}(w_0),L^{p_1}(w_1)]_\theta=L^p(w),
\end{equation*}
where
\begin{equation*}
  \frac{1}{p}=\frac{1-\theta}{p_0}+\frac{\theta}{p_1},\qquad
  w^{\frac{1}{p}}=w_0^{\frac{1-\theta}{p_0}}w_1^{\frac{\theta}{p_1}}.
\end{equation*}
\end{lemma}

\begin{proof}
Note that the choice of $\theta\in(0,1)$ determines both
\begin{equation*}
  p_0=p_0(\theta)=\frac{1-\theta}{\frac{1}{p}-\frac{\theta}{p_1}},\qquad
  w_0=w_0(\theta)=w^{\frac{p_0}{p(1-\theta)}}w_1^{-\frac{p_0\cdot\theta}{p_1(1-\theta)}},
\end{equation*}
so it remains to check that we can choose $\theta\in(0,1)$ so that $p_0\in(1,\infty)$ and $w_0\in A_{p_0}^{\zeta_0}(\varphi)$ for some $\zeta_0>0$. Since $p_0(0)=p\in(1,\infty)$, the first condition is obvious for small enough $\theta>0$ by continuity.

To check that $w_0\in A_{p_0}^{\zeta_0}(\varphi)$ for some $\zeta_0>0$, we consider a cube $Q$ and write
\begin{equation*}
\begin{split}
  &\ave{w_0}_Q \ave{w_0^{-\frac{1}{p_0-1}}}_Q^{p_0-1}
  =\ave{w^{\frac{p_0}{p(1-\theta)}}w_1^{-\frac{p_0\cdot\theta}{p_1(1-\theta)}}}_Q
  \ave{w^{-\frac{p_0'}{p(1-\theta)}}w_1^{\frac{p_0'\cdot\theta}{p_1(1-\theta)}}}_Q^{p_0-1}  \\
  &=\ave{w^{\frac{p_0}{p(1-\theta)}}(w_1^{-\frac{1}{p_1-1}})^{\frac{p_0\cdot\theta}{p_1'(1-\theta)}}}_Q
  \ave{(w^{-\frac{1}{p-1}})^{\frac{p_0'}{p'(1-\theta)}}w_1^{\frac{p_0'\cdot\theta}{p_1(1-\theta)}}}_Q^{p_0-1},
\end{split}
\end{equation*}
where $q':=q/(q-1)$ denotes the conjugate exponent of $q\in\{p,p_0,p_1\}$.

In the first average, we use H\"older's inequality with exponents $1+\eps^{\pm 1}$, and in the second with exponents $1+\delta^{\pm 1}$ to get
\begin{equation}\label{eq:beforeRHI3}
\begin{split}
  &\leq \ave{w^{\frac{p_0(1+\eps)}{p(1-\theta)}}}_Q^{\frac{1}{1+\eps}}  \ave{(w_1^{-\frac{1}{p_1-1}})^{\frac{p_0\theta(1+\eps)}{p_1'\eps(1-\theta)}}}_Q^{\frac{\eps}{1+\eps}}
  \ave{(w^{-\frac{1}{p-1}})^{\frac{p_0'(1+\delta)}{p'(1-\theta)}} }_Q^{\frac{p_0-1}{1+\delta}}  \\ 
  &\qquad\times\ave{ w_1^{\frac{p_0'\theta(1+\delta)}{p_1\delta(1-\theta)}}}_Q^{\frac{\delta(p_0-1)}{1+\delta}}  \\
  &=\ave{w^{r(\theta)}}_Q^{\frac{1}{1+\eps}}\ave{(w_1^{-\frac{1}{p_1-1}})^{s(\theta)}}_Q^{\frac{\eps}{1+\eps}}
  \ave{(w^{-\frac{1}{p-1}})^{t(\theta)} }_Q^{\frac{p_0-1}{1+\delta}}  \\
  &\qquad\times\ave{ w_1^{u(\theta)}}_Q^{\frac{\delta(p_0-1)}{1+\delta}},
\end{split}
\end{equation}
where
\begin{equation*}
  r(\theta):=\frac{p_0(\theta)(1+\eps)}{p(1-\theta)},\qquad
  s(\theta):=\frac{\theta p_0(\theta)(1+\eps)}{p_1'\eps(1-\theta)}
\end{equation*}
and
\begin{equation*}
  t(\theta):=\frac{p_0'(\theta)(1+\delta)}{p'(1-\theta)},\qquad
  u(\theta):=\frac{\theta p_0'(\theta)(1+\delta)}{p_1\delta(1-\theta)}.
\end{equation*}

Now, we choose $\eps=\frac{\theta p}{p_1'}$ and $\delta=\frac{\theta p'}{p_1}$ in a such a way that
\begin{equation*}
  r(\theta)=s(\theta)=\frac{p_0(\theta)(p_1'+\theta p)}{pp_1'(1-\theta)},
\end{equation*}
and
\begin{equation*}
  t(\theta)=u(\theta)=\frac{p_0(\theta)'(p_1+\theta p')}{p'p_1(1-\theta)}. 
\end{equation*}

The strategy to proceed is to use the reverse H\"older inequality for $A_v^{\tilde\zeta}(\varphi)$ weights due to Wu--Wang (see Proposition 15 in \cite{WW:2018}), which says that for each $W\in A_v^{\tilde\zeta}(\varphi)$ there exists $\eta>0$ such that
\begin{equation}\label{eq:RHI2}
  \ave{W^t}_Q^{1/t}\lesssim \ave{W}_{Q}\varphi(|Q|)^{\eta}
\end{equation}
for all $t\leq 1+\tilde\eta$ and for some $\tilde\eta>0$.

Recalling that $p_0(0)=p$, we see that $r(0)=t(0)=1$. By continuity, given any $\tilde\eta>0$, we find that 
\begin{equation}\label{eq:new weights}
  \max(r(\theta),t(\theta))\leq 1+\tilde\eta\,\,\,\text{for all small enough}\,\,\,\theta>0.
\end{equation}

Next we will apply another property of $A_v^{\tilde\zeta}(\varphi)$ weights as stated by Wu--Wang in Proposition 15 of \cite{WW:2018}, namely:
\newline If $1<v<\infty$, we have 
\begin{equation}\label{Ap(zeta)(phi) prop.}
  W\in A_v^{\tilde\zeta}(\varphi)\Longleftrightarrow W^{1-v'}\in A_{v'}^{\tilde\zeta}(\varphi),\qquad\text{where}\qquad\frac{1}{v}+\frac{1}{v'}=1.
\end{equation}

By \eqref{Ap(zeta)(phi) prop.} we have that $w\in A_p^\zeta(\varphi)$, $w^{-\frac{1}{p-1}}\in A_{p'}^\zeta(\varphi)$, $w_1\in A_{p_1}^{\zeta_1}(\varphi)$ and $w_1^{-\frac{1}{p_1-1}}\in A_{p_1'}^{\zeta_1}(\varphi)$. Hence by \eqref{eq:new weights} each of these four functions satisfies the reverse H\"older inequality \eqref{eq:RHI2} for all $t\leq 1+\tilde\eta$ and for some $\tilde\eta>0$. Thus, for all small enough $\theta>0$, we have
\begin{equation*}
\begin{split}
  \eqref{eq:beforeRHI3}
  &\lesssim \ave{w}_Q^{\frac{p_0}{p(1-\theta)}} \ave{w_1^{-\frac{1}{p_1-1}}}_Q^{\frac{\theta p_0}{p_1'(1-\theta)}}
  \ave{w^{-\frac{1}{p-1}} }_Q^{\frac{p_0'(p_0-1)}{p'(1-\theta)}}  \\
  &\qquad\times\ave{ w_1 }_Q^{\frac{\theta p_0'(p_0-1)}{p_1(1-\theta)}}\varphi(|Q|)^{\eta\frac{p_0}{p(1-\theta)}+\eta \frac{\theta p_0}{p_1'(1-\theta)}+\eta\frac{p_0'(p_0-1)}{p'(1-\theta)}+\eta\frac{\theta p_0'(p_0-1)}{p_1(1-\theta)}}  \\
  &=\ave{w}_Q^{\frac{p_0(\theta) }{p(1-\theta)}} \ave{w_1^{-\frac{1}{p_1-1}}}_Q^{\frac{\theta p_0(\theta)}{p_1'(1-\theta)}}\ave{w^{-\frac{1}{p-1}} }_Q^{\frac{p_0(\theta)}{p'(1-\theta)}}  \\
  &\qquad\times\ave{ w_1 }_Q^{\frac{\theta p_0(\theta)}{p_1(1-\theta)}}\varphi(|Q|)^{\eta\frac{p_0(\theta)}{p(1-\theta)}+\eta \frac{\theta p_0(\theta)}{p_1'(1-\theta)}+\eta\frac{p_0(\theta)}{p'(1-\theta)}+\eta\frac{\theta p_0(\theta)}{p_1(1-\theta)}} \\
  &=(\ave{w}_Q\ave{w^{-\frac{1}{p-1}} }_Q^{p-1})^{\frac{p_0(\theta) }{p(1-\theta)}}(\ave{ w_1 }_Q\ave{w_1^{-\frac{1}{p_1-1}}}_Q^{p_1-1})^{\frac{\theta p_0(\theta)}{p_1(1-\theta)}}  \\
  &\qquad\times\varphi(|Q|)^{\frac{p_0(\theta)(\eta+\eta\theta)}{1-\theta}} \\
  &\leq [w]_{A_{p}^{\zeta}(\varphi)}^{\frac{p_1}{p_1-\theta p}}[w_1]_{A_{p_1}^{\zeta_1}(\varphi)}^{\frac{\theta p}{p_1-\theta p}}\varphi(|Q|)^{\zeta_0 p_0(\theta)},
\end{split}
\end{equation*}
where $\zeta_0=\frac{\eta+\eta\theta+\zeta+\zeta_1\theta}{1-\theta}>0$. In combination with the lines preceding \eqref{eq:beforeRHI3}, we have shown that
\begin{equation*}
  [w_0]_{A_{p_0}^{\zeta_0}(\varphi)}\lesssim [w]_{A_{p}^{\zeta}(\varphi)}^{\frac{p_1}{p_1-\theta p}} [w_1]_{A_{p_1}^{\zeta_1}(\varphi)}^{\frac{\theta p}{p_1-\theta p}}<\infty,
\end{equation*}
provided that $\theta>0$ is small enough. This concludes the proof. 
\end{proof}

We now have the last missing ingredient of the proof of Theorem \ref{thm:Ap(varphi) compact}:

\begin{proof}[Proof of Proposition \ref{prop:Ap(varphi)}]
We are given $q,q_1\in (1,\infty),\zeta,\zeta_1>0$, and weights $v\in A_q^\zeta(\varphi)$, $v_1\in A_{q_1}^{\zeta_1}(\varphi)$. By Lemma \ref{lem:main3}, there is some $q_0\in (1,\infty),\zeta_0>0$, a weight $v_0\in A_{q_0}^{\zeta_0}(\varphi)$, and $\theta\in(0,1)$ such that
\begin{equation*}
  \frac{1}{q}=\frac{1-\theta}{q_0}+\frac{\theta}{q_1},\qquad
  v^{\frac{1}{q}}=v_0^{\frac{1-\theta}{q_0}}v_1^{\frac{\theta}{q_1}}.
\end{equation*}
By Theorem \ref{thm:SW}, we then have $L^q(v)=[L^{q_0}(v_0),L^{q_1}(v_1)]_\theta$, as we claimed.
\end{proof}

By combining Theorem \ref{thm:CwKa}, Lemma \ref{lem:lemOk} and Proposition \ref{prop:Ap(varphi)} we can prove Theorem \ref{thm:Ap(varphi) compact} as follows:

\begin{proof}[Proof of Theorem \ref{thm:Ap(varphi) compact}]
Recall that the assumptions of Theorem \ref{thm:Ap(varphi) compact} are in force. In particular, $T$ is a bounded linear operator on $L^p(w)$ for all $p\in(1,\infty)$, all $w\in A_p^\zeta(\varphi)$ and all $\zeta>0$. In addition, it is assumed that $T$ is a compact operator on $L^{p_1}(w_1)$ for some $p_1\in(1,\infty)$, some $w_1\in A_{p_1}^{\zeta_1}(\varphi)$ and some $\zeta_1>0$. We need to prove that $T$ is actually compact on $L^p(w)$ for all $p\in(1,\infty)$, all $w\in A_p^\zeta(\varphi)$ and all $\zeta>0$. By Proposition \ref{prop:Ap(varphi)}, we have
\begin{equation*}
  L^p(w)=[L^{p_0}(w_0),L^{p_1}(w_1)]_\theta
\end{equation*}
for some $p_0\in(1,\infty)$, some $\zeta_0>0$, some $w_0\in A_{p_0}^{\zeta_0}(\varphi)$, and some $\theta\in(0,1)$. Writing $X_j=Y_j=L^{p_j}(w_j)$, we know that $T:X_0+X_1\to Y_0+Y_1$, that $T:X_j\to Y_j$ is bounded, and that $T:X_1\to Y_1$ is compact (since the last two assertions were assumed). By Lemma \ref{lem:lemOk}, the last condition \eqref{it:lattice} of Theorem \ref{thm:CwKa} is also satisfied by these spaces $X_j=L^{p_j}(w_j)$. By Theorem \ref{thm:CwKa}, it follows that $T$ is also compact on $[X_0,X_1]_\theta=[Y_0,Y_1]_\theta=L^p(w)$.
\end{proof}

We provide an application of Theorem \ref{thm:Ap(varphi) compact} that concerns pseudo-different\-ial operators with smooth symbols.

\subsection{Commutators of pseudo-differential operators with smooth symbols}

Following \cite{Taylor}, we say that a symbol $\sigma$ belongs to $S_{1,\lambda}^m$ if $\sigma(x,\xi)$ is a smooth function of $(x,\xi)\in\R^d\times\R^d$ and satisfies the following estimate:
\begin{equation*}
  \abs{\partial_x^\mu\partial_\xi^\nu\sigma(x,\xi)}\lesssim(1+\abs{\xi})^{m-\abs{\nu}+\lambda|\mu|},
\end{equation*}
for all $\mu,\nu\in \N^d$, where $m\in\R$.

Let $\sigma(x,\xi)\in S_{1,\lambda}^m$. The pseudo-differential operator $T$ is defined by
\begin{equation*}
  Tf(x)=\int_{\R^d}\sigma(x,\xi)e^{2\pi ix\cdot\xi}\widehat f(\xi)d\xi,
\end{equation*}
where $f$ is a Schwartz function and $\widehat f$ denotes the Fourier transform of $f$. As usual, $L_{1,\lambda}^m$ will denote the class of pseudo-differential operators with symbols in $S_{1,\lambda}^m$.

Miller \cite{Miller} showed the boundedness of $L_{1,0}^0$ pseudo-differential operators on $L^p(w)$ for $1<p<\infty$ and $w\in A_p(\R^d)$. Tang \cite{Tang} improved the results of Miller by showing the boundedness of $L_{1,0}^0$ pseudo-differential operators and their commutators on $L^p(w)$, where $w\in A_p^\zeta(\varphi)$, $\varphi(t)=1+t$ and $\zeta>0$ (Tang also makes a remark about the case $L_{1,\lambda}^0$ ($0<\lambda<1$); see \cite[after Corollary 1.2]{Tang}).

We will apply Theorem \ref{thm:Ap(varphi) compact} to the commutators of pseudo-differential operators $T\in L_{1,0}^0$. We need the following result of Tang \cite{Tang}:

\begin{theorem}[\cite{Tang}, Theorem 1.2]\label{thm:bdd. of comm. of pseudo-diff. oper. with smooth symbols}
Suppose that $T\in L_{1,0}^0$ and let $b\in\BMO(\R^d)$, for $1<p<\infty$. Then $[b,T]$ is bounded on $L^p(w)$ for all $w\in A_p^\zeta(\varphi)$, where $\varphi(t)=1+t$ and $\zeta>0$.
\end{theorem}
 
As explained in Remark \ref{rmk:pseudodiff} these operators are instances of Calder\'on--Zygmund operators and thus they satisfy the assumption of Uchiyama's Theorem \ref{thm:Uchi}. Therefore, we can combine Theorems \ref{thm:Uchi}, \ref{thm:Ap(varphi) compact} and \ref{thm:bdd. of comm. of pseudo-diff. oper. with smooth symbols} in order to recover a very recent result of Guo--Zhou \cite{GZ}:

\begin{theorem}[\cite{GZ}, Theorem 2.1] Suppose that $T\in L_{1,0}^0$ and let $b\in CMO(\R^d)$, for $1<p<\infty$. Then the commutator $[b,T]$ is a compact operator on $L^p(w)$ for all $w\in A_p^\zeta(\varphi)$, where $\varphi(t)=1+t$ and $\zeta>0$.
\end{theorem}

\begin{proof}
We verify the assumptions of Theorem \ref{thm:Ap(varphi) compact} for $[b,T]$ in place of $T$: By Theorem \ref{thm:bdd. of comm. of pseudo-diff. oper. with smooth symbols} $[b,T]$ is a bounded operator on $L^p(w)$ for all $1<p<\infty$, all $w\in A_p^\zeta(\varphi)$ and all $\zeta>0$. By Theorem \ref{thm:Uchi}, $[b,T]$ is a compact operator on $L^{p_1}(\R^d)=L^{p_1}(w_1)$ for any $1<p_1<\infty$ with $w_1\equiv 1\in A_{p_1}^{\zeta_1}(\varphi)$ and any $\zeta_1>0$. Thus Theorem \ref{thm:Ap(varphi) compact} applies to give the compactness of $[b,T]$ on $L^p(w)$ for all $p\in (1,\infty)$, all $w\in A_p^\zeta(\varphi)$ and all $\zeta>0$.
\end{proof}

As in the case of the commutators of fractional integral operators in Section \ref{comm. fr. int. op.} the original proof in \cite{GZ} relied on verifying the weighted Fr\'echet--Kolmogorov criterion \cite{ClopCruz}, which is avoided by the argument above.

\subsection*{Acknowledgements} The second author wishes to thank his doctoral supervisor Prof. Tuomas Hyt\"onen for helpful discussions.

\subsection*{Funding} Both authors were supported by the Academy of Finland through the Grant no. 314829. The second author would like to thank the Foundation for Education and European Culture (Founders Nicos and Lydia Tricha), Greece for their financial support.


\end{document}